\documentclass[a4paper,12pt,reqno]{amsart}
\usepackage{amssymb}
\usepackage{amsmath}
\usepackage{enumitem}
\usepackage{mathrsfs}
\usepackage[all]{xy}
\usepackage{mathtools}
\usepackage{hyperref}
\usepackage{url}
\usepackage{bbm}

\usepackage[OT2,T1]{fontenc}
\DeclareSymbolFont{cyrletters}{OT2}{wncyr}{m}{n}
\usepackage[british]{babel}

\usepackage[margin=1.3in]{geometry}
\numberwithin{equation}{section} \numberwithin{figure}{section}

\DeclareMathOperator{\Gal}{Gal} 
 
\DeclareMathOperator{\Spec}{Spec}
 
 \DeclareMathOperator{\re}{Re}
   
\DeclareMathOperator{\vol}{vol} 
 \DeclareMathOperator{\Val}{Val}

\DeclareMathOperator{\Br}{Br} 

\DeclareMathOperator{\inv}{inv} \DeclareMathOperator{\res}{\partial}

\DeclareMathOperator{\Norm}{N} \DeclareMathOperator{\Frob}{Frob}
 \DeclareMathOperator{\dens}{dens}
\DeclareMathOperator{\HH}{H}
\DeclareMathOperator{\lcm}{lcm}
\DeclareMathOperator{\Li}{Li}

\DeclareSymbolFont{cyrletters}{OT2}{wncyr}{m}{n}
\DeclareMathSymbol{\Sha}{\mathalpha}{cyrletters}{"58}
\DeclareMathSymbol{\Be}{\mathalpha}{cyrletters}{"42}

\newcommand{\OO}{\mathcal{O}}

\newcommand{\xx}{\mathbf{x}}
\newcommand{\yy}{\mathbf{y}}

\renewcommand{\1}{\mathbf{1}}

\newcommand\FF{\mathbb{F}}
\newcommand\PP{\mathbb{P}}
\newcommand\ZZ{\mathbb{Z}}
\newcommand\NN{\mathbb{N}}
\newcommand\QQ{\mathbb{Q}}
\newcommand\RR{\mathbb{R}}
\newcommand\CC{\mathbb{C}}
\newcommand\GG{\mathbb{G}}

\newcommand\Gm{\GG_\mathrm{m}}
\newcommand{\Adele}{\mathbf{A}}
\newcommand{\br}{\mathscr{B}}

\newcommand{\fK}{\mathfrak{K}}

\newcommand{\bv}{\boldsymbol{v}}

\newcommand{\eps}{\varepsilon}

\newcommand{\fp}{\mathfrak{p}}
\newcommand{\fq}{\mathfrak{q}}
\newcommand{\fa}{\mathfrak{a}}

\newcommand{\bx}{\mathbf{x}}
\newcommand{\by}{\mathbf{y}}
\newcommand{\ba}{\boldsymbol{a}}

\newcommand{\Mod}[1]{\;(\operatorname{mod}\,#1)}

\newtheorem{lemma}{Lemma}

\newtheorem{theorem}[lemma]{Theorem}
\newtheorem{proposition}[lemma]{Proposition}
\newtheorem{corollary}[lemma]{Corollary}

\theoremstyle{definition}
\newtheorem{example}[lemma]{Example}
\newtheorem{definition}[lemma]{Definition}
\newtheorem{remark}[lemma]{Remark}

\numberwithin{lemma}{section}

\usepackage{color}

\newcommand{\dan}[1]{{\color{blue} \sf $\clubsuit\clubsuit\clubsuit$ Dan: [#1]}}

\begin{document}

\title[Frobenian multiplicative functions and rational points]{Frobenian multiplicative functions and rational points in fibrations}

\author{Daniel Loughran}
\address{
Department of Mathematical Sciences \\
University of Bath \\
Claverton Down \
Bath\\ 
BA2 7AY\\
UK.}

\author{\sc Lilian Matthiesen}
\address{Lilian Matthiesen \\
KTH\\
Department of Mathematics\\
Lindstedtsv\"{a}gen 25\\
10044 Stockholm\\
Sweden}
\email{lilian.matthiesen@math.kth.se}


\begin{abstract}
	We consider the problem of counting the number of varieties in a family over $\QQ$ 
	with a rational point. We obtain lower bounds for this counting problem for  some families over $\PP^1$, even if the Hasse principle fails.
	We also obtain sharp results for some multinorm equations and 
	for specialisations of certain Brauer group elements on higher dimensional
	projective spaces, where we answer some cases of a question of Serre. 
	Our techniques come from arithmetic geometry and additive combinatorics.
\end{abstract}

\maketitle

\thispagestyle{empty}

\tableofcontents

\section{Introduction} \label{sec:intro}

\subsection{Rational points in fibrations} \label{sec:P1}

Let $V$ be a smooth projective variety over $\QQ$ equipped with a dominant morphism $\pi:V \to \PP^n$.
In this paper we are interested in the function
\begin{equation} \label{def:N}
	N(\pi,B) =\#\{x \in \PP^n(\QQ): H(x) \leq B, x \in \pi(V(\QQ))\}
\end{equation}
which counts the number of varieties in the family which have a rational point.
Here $H$ denotes the usual naive height on projective space, defined via
$$H(x_0:\cdots:x_n) = \max\{|x_0|,\ldots, |x_n|\},$$
whenever $(x_0,\ldots,x_n)$ is a primitive integer vector.
Such counting functions have been studied by numerous authors in recent times \cite{Bha14b,BBL15,Lou13, LS16, Ser90,Sof16}.
In \cite{LS16}, upper bounds were obtained for the closely related function
\begin{equation} \label{def:N_loc}
	N_{\mathrm{loc}}(\pi,B) =\#\{x \in \PP^n(\QQ): H(x) \leq B, x \in \pi(V(\Adele_\QQ))\}
\end{equation}
which counts the number of varieties in the family which are everywhere locally soluble. (Here $\Adele_\QQ$ denotes the adeles of $\QQ$.)
These upper bounds take the shape
\begin{equation} \label{eqn:LS}
	N_{\mathrm{loc}}(\pi,B) \ll \frac{B^{n+1}}{(\log B)^{\Delta(\pi)}}
\end{equation}
for a certain factor $\Delta(\pi)$. This generalised work of Serre \cite{Ser90}, which applied to the special case where the generic fibre of $\pi$ is a conic.

Both Serre \cite{Ser90} and the authors of \cite{LS16} asked whether this upper bound is sharp, under the necessary assumption that the set being counted is non-empty. In fact, since the singular locus forms a Zariski closed subset of $\PP^n$, for the bound \eqref{eqn:LS} to be sharp there must be a \emph{smooth} fibre which is everywhere locally soluble.

In this paper we prove numerous results answering this in the affirmative. To state our most general results, we require various conditions on the singular fibres of $\pi$ in terms of the Galois action on the irreducible components. This terminology (split/non-split/pseudo-split/non-pseudo-split fibre), as well as the definition of $\Delta(\pi)$, is recalled in \S\ref{sec:notation}.
Our first result is the following.

\begin{theorem} \label{thm:Serre}
	Let $V$ be a smooth projective variety over $\QQ$ with a morphism $\pi:V \to \PP^1$
	whose generic fibre is geometrically integral.
	Assume that each fibre of $\pi$ contains an irreducible component of multiplicity $1$ and that each
	non-pseudo-split fibre of $\pi$ lies over a rational point.
	Assume that there is a fibre of $\pi$ over some rational point which is smooth and everywhere locally soluble.
	Then
	$$N_{\mathrm{loc}}(\pi,B) \asymp \frac{B^2}{(\log B)^{\Delta(\pi)}}.$$
\end{theorem}

Note that we make no geometric assumptions on the smooth members of the family. In this generality it is the first result of its kind in the literature which gives sharp lower bounds when $\Delta(\pi) > 0$. The case $\Delta(\pi) = 0$ is proved in \cite[Thm.~1.3]{LS16}, and all other results in the literature with $\Delta(\pi) > 0$ concern special classes of varieties, e.g.~conics or norm forms.

Let us give some special cases highlighting our result.
If the smooth fibres of $\pi$ satisfy the Hasse principle (i.e.~have a rational point as soon as they have an adelic point), then we deduce results about the counting function \eqref{def:N}.
As conics satisfy the Hasse principle, this gives the following.
(Note that for conic bundles, a fibre is pseudo-split if and only if it is split.)

\begin{corollary} \label{cor:conics}
	Let $\pi:V \to \PP^1$ be a non-singular conic bundle all of whose non-split fibres
	lie over rational points. 	Assume that there is a smooth fibre
	with a rational point.
	Then
	$$N(\pi,B) \asymp \frac{B^2}{(\log B)^{\Delta(\pi)}}.$$
\end{corollary}

This answers many new cases of Serre's question posed in \cite{Ser90}. 	
This result was previously only known when there are at most $3$ non-split fibres, as a special case of  \cite{Sof16}. 
As an explicit example, we have the following.

\begin{example} \label{ex:explicit_conics}
	Let $a \in \QQ^*\setminus \QQ^{*2}$, let $L_1,\ldots, L_r \in \ZZ[t]$ be linear polynomials
	whose homogenisations are pairwise linearly independent
	and let
	\begin{equation} \label{eqn:norm}
		W: \quad x^2 - ay^2 = L_1(t)\cdots L_r(t) z^2 \subset \mathbb{P}^{2} \times \mathbb{A}^1
	\end{equation}
	equipped with the projection $\pi: W \to \mathbb{A}^1$.
	Assume that there is some $t \in \QQ$ with $L_1(t)\cdots L_r(t) \neq 0$
	such that the conic in \eqref{eqn:norm} has  a rational point.
	Applying Corollary~\ref{cor:conics} to a smooth compactification of $W$
	yields
	$$N(\pi,B) \asymp 
	\begin{cases}
		\frac{B^2}{(\log B)^{(r+1)/2}}, & \text{if $r$ is odd}, \\
		\frac{B^2}{(\log B)^{r/2}}, & \text{if $r$ is even}. \\
	\end{cases}
	$$		
\end{example}

\subsection{Multiple fibres} \label{sec:multiple_intro}

Theorem \ref{thm:Serre} contains the technical geometric assumption that each fibre of $\pi$ contains an 
irreducible component of multiplicity $1$. It turns out that this assumption is \emph{necessary} for
the conclusion to hold.

It was shown in \cite{CTSSD97} that as soon as there are many
double fibres, only finitely many fibres have a rational point. It seems to have not been noticed before
that the following stronger result in fact holds.

\begin{theorem} \label{thm:double}
	Let $V$ be a smooth projective variety over a number field $k$
	equipped with a morphism $\pi:V \to \PP^1$
	whose generic fibre is geometrically integral.
	Assume that $\pi$ has at least $6$ double fibres over $\bar{k}$.
	Then the set
	$$\{x \in \PP^1(k): x \in \pi(V(k_v)) \text{ for all } v \notin S\}$$
	is finite for any finite set of places $S$  of $k$.
\end{theorem}
This shows   there are only finitely many fibres which are everywhere locally soluble, and hence the conclusion of Theorem \ref{thm:Serre}  does not hold in this case.

\subsection{Controlling failures of the Hasse principle}
Theorem \ref{thm:Serre} counts the number of varieties in a family which are everywhere locally soluble. This gives results for rational points in families if the fibres satisfy the Hasse principle. However, in general the Hasse principle can fail and it is a great challenge to control failures of the Hasse principle in families. We are able to obtain results here providing that the Brauer--Manin obstruction controls such failures.

\begin{theorem} \label{thm:RP}
	Let $V$ be a smooth projective variety over $\QQ$ equipped with a morphism $\pi:V \to \PP^1$
	whose generic fibre is geometrically integral and rationally connected.
	Assume that each non-split fibre of $\pi$ lies over a rational point
	and that the Brauer--Manin obstruction is the only one to the Hasse principle for the fibres of $\pi$. If $V(\QQ) \neq \emptyset$ then
	$$N(\pi,B) \gg \frac{B^2}{(\log B)^{\omega(\pi)}}, \quad \mbox{for some } \omega(\pi)>0.$$
\end{theorem}

We prove this result by combining our techniques with the method of Harpaz and Wittenberg \cite{HW16} on the fibration method. Theorem \ref{thm:RP} and its proof may be viewed as a quantitative version of the results from \cite{HW16}. Recall that the Brauer--Manin obstruction is the only one to the Hasse principle for torsors under tori \cite{San81}. In particular, Theorem \ref{thm:RP} applies to the following multinorm equations.

\begin{corollary}
	Let $E=E_1 \times \dots \times E_s$ be a product of number fields,
	let $a_1,\dots,a_r \in \NN$, and let $L_1,\ldots, L_r \in \ZZ[x_1,\ldots,x_n]$ be linear polynomials
	whose homogenisations are pairwise linearly independent. Let $V$ be a smooth projective compactification of the variety
	$$
		W: \quad \Norm_{E/\QQ}(t) = L_1(x)^{a_1}\cdots L_r(x)^{a_r} \subset \mathbb{A}^{[E:\QQ]} \times \mathbb{A}^1,
	$$
	equipped with the projection $\pi:V \to \PP^1$ coming from the $x$-coordinate. Assume that $V(\QQ) \neq \emptyset$. Then
	$$N(\pi,B) \gg \frac{B^{2}}{(\log B)^{\omega(\pi)}}, \quad \mbox{for some } \omega(\pi) > 0.$$
\end{corollary}
Let us compare the assumptions in Theorem \ref{thm:Serre} and Theorem \ref{thm:RP}. Theorem~\ref{thm:RP} has the  condition that the generic fibre is rationally connected. We actually prove a more general version (Theorem \ref{thm:RP2}) which only poses cohomological conditions on the generic fibre; but in this more general result one obviously also needs to stipulate that there is an irreducible component of multiplicity $1$ in each fibre, which is automatic for families of rationally connected varieties \cite{GHS03} 

One subtle difference is that Theorem \ref{thm:Serre}  assumes that all \emph{non-pseudo-split} fibres lie over rational points, whereas  Theorem \ref{thm:RP} requires all \emph{non-split} fibres to lie over rational points. The latter condition is stronger in general.

\begin{example}
A difficult case, currently out of reach of \cite{HW16}, is
$$(x_1^2 - ax_2^2)(y_1^2 - by_2^2)(z_1^2 - abz_2^2) = f(t),$$
where $f$ is an irreducible polynomial of large degree and none of $a,b,ab$ is a square in the residue field of $f$. Here every fibre is pseudo-split, so that $\Delta(\pi) = 0$, and moreover every rational fibre is even everywhere locally soluble \cite[Prop.~5.1]{CT14}. Despite this, the fibres can still fail the Hasse principle, and it is not known whether the Brauer--Manin obstruction is the only one for the total space. This example is covered by Theorem \ref{thm:Serre} but not by Theorem~\ref{thm:RP}. 
\end{example}

\begin{remark}
	The proof of Theorem \ref{thm:RP} gives an explicit value for $\omega(\pi)$
	(cf.~Remark \ref{rem:admissible}), which we doubt is sharp
	in general. Improving this would require significant new ideas 
	on the version of the fibration method developed in \cite{HW16}.
\end{remark}

\subsection{Frobenian multiplicative functions}
We prove Theorem \ref{thm:Serre} by finding sufficient conditions for a fibre to be everywhere locally
soluble. These conditions allow us to reduce to studying the average orders
of certain multiplicative functions evaluated at linear forms. The multiplicative functions which arise this way are built out of data coming from the splitting behaviour of primes in number fields and are closely related to Serre's notion of \emph{frobenian functions} \cite[\S3.3]{Ser12}.

In this paper we introduce a class of multiplicative functions, called \emph{frobenian multiplicative functions}, which includes the above.
The divisor function, the (normalised) sums of two squares function and the indicator function
for sums of two squares are frobenian multiplicative functions. Other examples are  indicator functions for numbers
all of whose prime factors are congruent to $ a \bmod q$, for some fixed $a,q$, and the reduction modulo $q$ of the $n$th coefficient of a Hecke eigenform \cite[\S 3.4.3]{Ser12} (suitably considered as a complex number). Definitions and further details on this class, including the definition of the mean value $m(\rho)$, can be found in \S\ref{sec:frob}, see specifically Definitions \ref{def:frob}, \ref{def:e-weak_frobenian}, and \ref{def:class_F}.
Our main analytic result concerning these is the following.

\begin{theorem} \label{thm:frob}
	Let $\rho_1,\ldots, \rho_r$ be real-valued non-negative frobenian multiplicative functions with $m(\rho_j) \neq 0$, 
	and extend each function $\rho_j$ to all of $\ZZ$ by setting $\rho_j(-m)=0$ for all $m \geq 0$.
	Let $L_1(x_0,\ldots,x_n),\ldots, L_r(x_0,\ldots,x_n) \in \ZZ[x_0,\ldots,x_n]$ be
	linear polynomials whose non-constant parts are pairwise linearly independent. 
	Let $\mathfrak{K} \subset [-1,1]^{n+1}$ be convex subset of positive measure and $\boldsymbol{a} \in \RR^{n+1}$. Then there exists $C_{\mathfrak{K}, \boldsymbol{\rho}, \mathbf{L}} \geq 0$ such that
	$$\sum_{\xx \in (B\mathfrak{K} + \boldsymbol{a}) \cap \ZZ^{n+1}} \rho_1(L_1(\xx)) \cdots \rho_r(L_r(\xx))
	= ( C_{\mathfrak{K}, \boldsymbol{\rho}, \mathbf{L}} + o(1))
	B^{n+1}\prod_{j=1}^r(\log B)^{m(\rho_j)-1} ,$$
	as $B \to \infty$. Moreover, we have $C_{\mathfrak{K}, \boldsymbol{\rho}, \mathbf{L}} > 0$ 
	if and only if there exists $\xx \in \ZZ^{n+1}$ 
	with $\rho_1(L_1(\xx)) \cdots \rho_r(L_r(\xx)) > 0$
	and $\yy \in \fK$ with $L_j(\yy) > L_j(\boldsymbol{0})$ for all $1 \leq j \leq r$.
\end{theorem}
\begin{remark}
 The first of the two conditions for positivity of $C_{\mathfrak{K}, \boldsymbol{\rho}, \mathbf{L}}$ 
 is clearly necessary.
 The second condition ensures that the linear parts of the polynomials $L_j$ are 
 simultaneously positive at some point in $\mathfrak{K}$, which is necessary in order
 for the polynomials $L_j$ to be simultaneously positive on a positive proportion of 
 lattice points in $B\mathfrak{K} + \boldsymbol{a}$ for all sufficiently large $B$.
\end{remark}

The condition $m(\rho) \neq 0$ on the means is clearly necessary; this rules out trivial cases such as the multiplicative function $\rho$ with  $\rho(n) = 1$ if and only if $n=1$.

\begin{remark}[] \label{rem:const}
	An explicit expression for the leading constant in Theorem \ref{thm:frob}
	can be found in Remark \ref{rem:leading-constant}.
	It is given as an alternating sum of Euler products, which is important for the following reason:
	
	There can naturally be \emph{local obstructions} to the positivity of this constant.
	E.g.~if $\rho$ is a frobenian multiplicative function with $\rho(2^n) = 0$
	then we have $\rho(x)\rho(y)\rho(x + y) = 0$ for all integers $x,y$. Here there is an obstruction at $2$.
		
	But there can also be \emph{global obstructions} to positivity of the leading constant,
	which are not explained by any local obstructions. This is reflected in the fact
	that the leading constant is not an Euler product in general.
	An explicit example of this kind can be found in Remark~\ref{rem:const_BM}. (This comes from a Brauer--Manin obstruction to the Hasse principle on some auxiliary variety.)
	
	The fact that the leading constant is not an Euler product means that we can obtain asymptotic results even in situations when weak approximation fails, which is unusual for this type of result.
\end{remark}

Theorem \ref{thm:frob} is proved using tools from additive combinatorics.
In particular, we build on new results about multiplicative functions from Matthiesen \cite{Mat14, Mat16}.
Crucially, the correct order lower bound in Theorem \ref{thm:Serre} as well as the asymptotic result in Theorem \ref{thm:frob} require the full strength of these new results on multiplicative functions from \cite{Mat14, Mat16}. 
In particular, our results presented here are out of reach from the additive combinatorics methods established or used in
\cite{GT-linearprimes, BMS14, HWS, BM17, M-squarefree, HW16}.

\subsection{Higher-dimensional bases} 
There are two main difficulties in generalising Theorem \ref{thm:Serre} to families of varieties over $\PP^n$ with $n > 1$. Firstly, Theorem \ref{thm:frob} takes care of codimension $1$ behaviour, but in general there could be higher codimension behaviour to deal with. To overcome this one would need
to combine our techniques with some suitable version of the sieve of Ekedahl (see e.g.~\cite{Eke91},\cite[\S3]{Bha14}, \cite[\S3]{BBL15}).
The second problem
is that frobenian multiplication functions are built out of data concerning number fields, whereas
over higher-dimensional bases one would also need to deal with finitely generated extensions of $\QQ$.

We can overcome these problems in two cases: for  Serre's problem
\cite{Ser90} regarding specialisations of Brauer group elements and for certain multinorm equations.

\subsubsection{Specialisations of Brauer group elements}
Let $U$ be a smooth variety over $\QQ$ equipped with a height function $H$ and let $\br \subset \Br U$ be a finite subgroup.
Then Serre \cite{Ser90} proposed to study the \emph{zero-locus}
$$U(\QQ)_\br = \{ x \in U(\QQ): b(x) = 0 \in \Br \QQ \,\, \forall \, b \in \br\}$$
of $\br$, as well as the associated counting function
$$N(U,\br,B) = \# \{ x \in U(\QQ)_\br: H(x) \leq B\}.$$
This  problem  may be interpreted
more geometrically via families of Brauer-Severi varieties (e.g.~families of conics as we have already
met in \S\ref{sec:P1}), but one obtains a cleaner framework by working
with Brauer group elements directly. Serre achieved upper bounds for the counting problem in  the special case
where $U \subset \PP^n_\QQ$ and $|\br| = 2$, and asked whether his bounds were sharp. (Serre's upper
bounds were subsequently generalised in \cite[\S 5.3]{LS16} to arbitrary finite $\br \subset \Br U$).

Here the first issue mentioned above, regarding higher-codimension behaviour, does  not occur as Grothendieck's
purity theorem implies that only codimension $1$ is relevant. The second issue, regarding finitely
generated extensions of $\QQ$, disappears if one only considers algebraic Brauer group elements, i.e.~those
Brauer group elements which trivialise after a finite extension of $\QQ$.

Our result for Brauer groups is the following. (We recall in \S\ref{sec:Brauer_basic} various facts about Brauer groups, including the residue map $\partial_D$ appearing
in Theorem \ref{thm:Brauer}.)

\begin{theorem} \label{thm:Brauer}
	Let $U \subset \PP^n$ be a Zariski open subset which is the complement 
	of finitely many hyperplanes. Let $\br \subset \Br_1 U$ be a finite subset such that $U(\QQ)_\br \neq \emptyset$.
	Then as $B \to \infty$ we have
	$$N(U,\br,B) \asymp \frac{B^{n+1}}{(\log B)^{\Delta(\br)}},
	\quad \text{where }
	\Delta(\br) = \sum_{D \in (\PP^n)^{(1)}} \left(1 - \frac{1}{|\langle \partial_D \br \rangle|}\right).$$
\end{theorem}

Using the relationship between quaternion algebras and conics, Theorem \ref{thm:Brauer}
gives the following result.

\begin{corollary} \label{cor:conics_P^n}
	Let $\pi:V \to \PP^n$ be a non-singular conic bundle over $\QQ$ with a smooth
	fibre which contains a rational point.
	Assume that $\pi$ admits a rational section
	over a finite extension of $\QQ$ and that $\pi$ is smooth outside the complement of finitely many
	hyperplanes in $\PP^n$.
	Then
	$$N(\pi,B) \asymp \frac{B^{n+1}}{(\log B)^{\Delta(\pi)}}.$$
\end{corollary}

\subsubsection{Multinorm  equations}
We are also able to overcome the issues in higher dimension for some explicit families of 
multinorm equations.

\begin{theorem} \label{thm:norms}
	Let $E=E_1 \times \dots \times E_s$ be a product of number fields,
	let $a_1,\dots,a_r \in \NN$, and let $L_1,\ldots, L_r \in \ZZ[x_1,\ldots,x_n]$ be linear polynomials
	whose homogenisations are pairwise linearly independent. Let $V$ be a smooth projective model of the variety
	\begin{equation} \label{eqn:multinorm}
		W: \quad \Norm_{E/\QQ}(t) = L_1(\bx)^{a_1}\cdots L_r(\bx)^{a_r} \subset \mathbb{A}^{[E:\QQ]} \times \mathbb{A}^n,
	\end{equation}
	equipped with the projection $\pi:V \to \PP^n$ coming from the $\bx$-coordinate. Assume that $V$ has a smooth fibre which is everywhere locally soluble. Then
	$$N_{\mathrm{loc}}(\pi,B) \asymp \frac{B^{n+1}}{(\log B)^{\Delta(\pi)}}.$$
\end{theorem}

If $E/\QQ$ satisfies the Hasse norm principle, then we obtain a result for rational points.
This holds for example if $s=1$ and $E_1/ \QQ$ is cyclic (Hasse norm theorem), or  $s =2$ and the Galois closures of $E_1$ and $E_2$ are linearly
disjoint \cite{PR13}; see also \cite{DW14} for related results and references.
Taking $E/\QQ$ a quadratic field extension, $n=1$ and $a_j = 1$,
we recover Example~\ref{ex:explicit_conics}.

\subsection{Methodology and structure of the paper}
In \S \ref{sec:frob} we study the basic properties of frobenian multiplicative functions. In \S \ref{sec:Lilian} we prove our main analytic result (Theorem \ref{thm:frob}) using tools from additive combinatorics, from which all other counting results in this paper will be derived.

Theorem \ref{thm:Serre} is proved in \S \ref{sec:Serre}. The key new idea is to construct frobenian multiplicative functions which can be used to detect whether a fibre is everywhere locally soluble. Such detectors have previously only been constructed for special classes of varieties, e.g.~families of conics. Once we have these functions, the result follows from a suitable application of Theorem \ref{thm:frob}.

In  \S \ref{sec:RP} we prove Theorem \ref{thm:RP}. The proof is based on the proof of Theorem \ref{thm:Serre}, but much more subtleties arise coming from having to control the Brauer--Manin obstruction. We do this using ideas of Harpaz and Wittenberg \cite{HW16}.

In \S \ref{sec:general} we generalise our detector functions to pencils which may have non-split fibres over non-rational points. This construction is not required for our proofs, but is included to assist with future generalisations of our work.

The results concerning higher dimensional bases are proved in \S \ref{sec:Brauer} and \S\ref{sec:norms}. We finish in \S \ref{sec:multiple} with the proof of Theorem \ref{thm:double}.

\subsection{Notation and terminology} \label{sec:notation}
For an abelian group $A$ and a prime $\ell$, we denote its $\ell$-primary torsion subgroup by $A\{\ell\}$.

A variety is an integral separated finite type scheme over a field.
For a point $x$ of a scheme $X$, we denote by $\kappa(x)$ its residue field. 
All cohomology is taken with respect to the \'etale topology.

We denote by $\Val(k)$ the set of all non-archimedean places of a number field $k$. For a prime $p$ we denote by $v_p$ the associated
$p$-adic valuation.

The notation $O, \ll,\gg$ have their standard meaning in analytic number theory (Landau and Vinogradov notation, respectively). We say that $f \asymp g$ if $f \ll g$ and $g \ll f$.

\begin{definition}
Let $k$ be a perfect field with algebraic closure $\bar{k}$ and $X$ a finite type scheme over $k$. The absolute Galois group $\Gal(\bar{k}/k)$ acts on the geometric irreducible components of $X$, i.e.~the irreducible components of $X \otimes_k \bar{k}$.
We say that $X$ is \emph{split} \cite[Def.~0.1]{Sko96} (resp.~\emph{pseudo-split} \cite[Def.~1.3]{LSS17}) if $\Gal(\bar{k}/k)$ (resp.~every element of $\Gal(\bar{k}/k)$) fixes some geometric irreducible component of multiplicity $1$.
\end{definition}

\begin{definition} \label{def:Delta}
	Let $\pi:V \to X$ be a dominant proper morphism of smooth irreducible varieties over a perfect
	field $k$. For each point
	$x \in X$, we choose some finite group $\Gamma_x$ through which 
	the absolute Galois group $\Gal(\overline{\kappa(x)}/ \kappa(x))$ 
	acts on the irreducible 	components of $\pi^{-1}(x)_{\overline{\kappa(x)}}:=\pi^{-1}(x) \otimes_{\kappa(x)} \overline{\kappa(x)}$.
	We define
	$$\delta_x(\pi) = \frac{\# \left\{ \gamma \in \Gamma_x : 
	\begin{array}{l}
		\gamma \text{ fixes an irreducible component} \\
		\text{of $\pi^{-1}(x)_{\overline{\kappa(x)}}$ of multiplicity } 1
	\end{array}
	\right \}}
	{\# \Gamma_x }.$$
	Let $X^{(1)}$ denote the set of codimension $1$ points of $X$. Then
	we let
	$$\Delta(\pi) = \sum_{D \in X^{(1)}} ( 1 - \delta_D(\pi)).$$
\end{definition}

\subsection*{Acknowledgements}
We are grateful to Tim Browning, Jean-Louis Colliot-Th\'{e}l\`{e}ne, J\"{o}rg Jahnel, David Harari, Dasheng Wei, and Olivier Wittenberg for helpful discussions. We are also indebted to Jean-Louis Colliot-Th\'{e}l\`{e}ne for help
with the proof of Theorem~\ref{thm:double}.
We are grateful to two anonymous referees for their comments on different parts of this paper.
Loughran is supported by  EPSRC grant EP/R021422/1 and UKRI Future Leaders Fellowship MR/V021362/1.
Matthiesen is supported by the Swedish Research Council Grant No.\ 2016-05198.

\section{Frobenian multiplicative functions} \label{sec:frob}

\subsection{Frobenian functions}
We begin by recalling some of the theory of \emph{frobenian functions}, following  Serre's treatment \cite[\S3.3]{Ser12}.
\begin{definition} \label{def:frob}
	Let $\rho: \Val(\QQ) \to \CC$ be a function. We say that $\rho$ is \emph{frobenian}
	if there exist 
	\begin{enumerate}
		\item[(a)] A finite Galois extension $K/\QQ$, with Galois group $\Gamma$;
		\item[(b)] A finite set of primes $S$ containing all the primes ramifying in $K$;
		\item[(c)] A class function $\varphi: \Gamma \to \CC$;
	\end{enumerate}
	such that for all $p \not \in S$ we have
	$$\rho(p) = \varphi(\Frob_p),$$
	where $\Frob_p \in \Gamma$ is the Frobenious element of $p$.
	We define the \emph{mean} of $\rho$ to be 
	$$m(\rho) = \frac{1}{|\Gamma|} \sum_{\gamma \in \Gamma}\varphi(\gamma).$$ 
\end{definition}
In the definition, recall that a \emph{class function} on a group $\Gamma$ is a function which is constant on the conjugacy classes of $\Gamma$. In particular $\varphi(\Frob_p)$ is well-defined, despite $\Frob_p$ only being well-defined up to conjugacy.

A subset of $\Val(\QQ)$ is called \emph{frobenian} if its indicator function is frobenian.
A basic example of a frobenian set is the set of all primes which are completely split in a finite extension $L/\QQ$. 

\begin{example}
	Let $\chi: \ZZ \to \CC^*$ be a Dirichlet character mod $n$. We claim that the function
	$p \mapsto \chi(p)$ is  frobenian.
	In the notation of Definition \ref{def:frob} one takes $K = \QQ(\zeta_n)$
	and $S = \{ p \mid n \}$, where $\zeta_n$ is a primitive $n$th root of unity.
	The map $\psi:(\ZZ/n\ZZ)^* \to \Gamma$ given by $m \mapsto (\zeta_n \mapsto \zeta_n^m)$
	is an isomorphism, and we have $\psi(p \bmod n) = \Frob_p$. 
	We then take $\varphi = \chi \circ \psi^{-1}$ and note that
	$\chi(p) = \varphi(\Frob_p)$ is thus frobenian.
\end{example}

There is an alternative way to view frobenian functions which makes it easier to relate different frobenian functions. Namely, let $G = \Gal(\bar{\QQ}/\QQ)$ and consider a frobenian function $\rho$ with associated class function $\varphi: \Gamma \to \CC$. Then we can write $\Gamma = G/N$ for some normal open subgroup $N$ and view $\varphi$ as an $N$-invariant class function on $G$. With this perspective, we equip $G$ and $N$ with their
	Haar probability measures, so that the quotient measure on $\Gamma$ is also the Haar
	probability measure. Then the mean of $\rho$ is 
	easily seen to be given by the formula
	\begin{equation} \label{eqn:m_N}
		m(\rho) = \int_{G} \varphi(g) \mathrm{d} g.
	\end{equation}
Using this one obtains the following.

\begin{lemma} \label{lem:product_frob}
	Let $\rho_1$ and $\rho_2$ be frobenian functions. Then $\rho_1 \cdot \rho_2$ is 
	also frobenian.
\end{lemma}
\begin{proof}
	Consider the associated finite sets of primes $S_i$ and $N_i$-invariant
	class functions $\varphi_i: G=\Gal(\bar{\QQ}/\QQ) \to \CC$ for $i \in \{1,2\}$.
	Then $\varphi_1\varphi_2$ is $(N_1 \cap N_2)$-invariant and 
	$(\rho_1\rho_2)(p) = (\varphi_1\varphi_2)(\Frob_p)$ for all $p \in S_1 \cup S_2$,
	where $\Frob_p$ denotes the frobenuis element of $G/(N_1 \cap N_2)$.
\end{proof}

\begin{lemma} \label{lem:primes}
	Let $\rho$ be a frobenian function with $m(\rho) \neq 0$. Then as $x \to \infty$:
	\begin{enumerate}
		\item $$\sum_{p \leq x} \rho(p) = m(\rho)\cdot \Li(x) + O\left(x \exp(-c \sqrt{\log x})\right), \quad \mbox{for some } c>0,$$
		where $\Li(x) = \int_{2}^\infty \mathrm{d} t/\log t$ denotes the logarithmic integral.
		\item $$\sum_{p \leq x} \frac{\rho(p)}{p} = m(\rho) \log \log x +C_{\rho} + O\Big(\frac{1}{\log x}\Big), \quad \mbox{for some constant } C_\rho.$$
		\item $$\sum_{p \leq x} \rho(p) \log p = m(\rho)\cdot x + O\left(x \exp(-c \sqrt{\log x})\right), \quad \mbox{for some } c>0.$$
		\item $$\prod_{\substack{p \leq x \\ |\rho(p)| < p  }} \left( 1 + \frac{\rho(p)}{p}\right) \sim C'_{\rho} (\log x)^{m(\rho)}, \quad \mbox{for some } C_\rho' \neq 0,$$
		where $C_\rho'$ is real and positive when $\rho$ is real-valued. 
	\end{enumerate}
\end{lemma}
\begin{proof}
	The first part is Serre's version of the Chebotarev density
	theorem \cite[Thm.~3.6]{Ser12}. The second and third part follow from partial summation. The fourth
	part follows from the second part on taking logs. Observe that the product in the fourth part runs over all but finitely many primes since $\rho$ is bounded.
\end{proof}

\subsubsection{Twisting by a Dirichlet character}

\begin{lemma}  \label{lem:finitely-many}
	Let $\rho$ be a frobenian function. 
	\begin{enumerate}
		\item Only finitely many primitive Dirichlet characters
	$\chi$ satisfy
	$m(\rho \chi)\neq 0.$	
	\end{enumerate}
	Assume that $\rho$ is real valued and non-negative and let $\chi$ be a 
	Dirichlet character.
	\begin{enumerate}
		\item[(2)] We have $|m(\rho\chi)| \leq m(\rho).$
		\item[(3)] The following are equivalent:
		\begin{enumerate}
			\item $|m(\rho\chi)| =m(\rho) $;
			\item $m(\rho \chi) = m(\rho)$;
			\item $\rho \chi(p) = \rho(p)$ 
			for all but finitely many primes $p$.
		\end{enumerate}
	\end{enumerate}
\end{lemma}
\begin{proof}
	First note that $\rho \chi$ is frobenian by Lemma \ref{lem:product_frob}.
	Let $\varphi: \Gamma \to \CC$ be a choice of class function associated to $\rho$,
	which we view as an $N$-invariant class
	function on $G=\Gal(\bar{\QQ}/\QQ)$ for some normal open subgroup $N$.
	Next, recall from class field theory that primitive Dirichlet characters are in one-to-one correspondence
	with continuous homomorphisms $G \to S^1$; namely the Artin map induces an isomorphism
	$\widehat{\mathbb{Z}}^*\cong G^{\mathrm{ab}}$, and primitive Dirichlet characters are
	exactly the characters of $\widehat{\mathbb{Z}}^*$. Let
	$\chi:G \to S^1$ be such a homomorphism, which by abuse of notation
	we identity with the corresponding primitive Dirichlet character.
	First assume that $\chi$ is non-trivial on $N$. Then by \eqref{eqn:m_N} we have
	\begin{align*}
		m(\rho\chi) &= \int_{G} \varphi(g) \chi(g) \mathrm{d} g 
		 = \frac{1}{|\Gamma|}\sum_{\gamma \in \Gamma} \int_{N} \varphi(\gamma n) \chi(\gamma n) \mathrm{d} n \\
		& = \frac{1}{|\Gamma|}\sum_{\gamma \in \Gamma} \varphi(\gamma) \chi(\gamma)\int_{N}  \chi(n) \mathrm{d} n
		= 0
	\end{align*}
	where the last line is by character orthogonality and the fact that $\chi$ is non-trivial on $N$.
	It follows that if $m(\rho \chi ) \neq 0$ then $\chi$ is trivial on $N$. But then $\chi$ is 
	just a character of $\Gamma$, of which there are only finitely many. This proves (1).
	
	For (2), note that $\varphi$ is also real and non-negative. We thus have
	\begin{equation} \label{eqn:re}
		|m(\rho\chi)| = \frac{1}{|\Gamma|} \left|\sum_{\gamma \in \Gamma} \varphi(\gamma) \chi(\gamma) \right| 
	\leq \frac{1}{|\Gamma|} \sum_{\gamma \in \Gamma} \varphi(\gamma) = m(\rho),
	\end{equation}
	as required, on using $|\chi| = 1$.
	To prove (3), we use the following fact:
	\begin{equation}
	\mbox{if $z_1,\dots,z_n \in \CC$ and $|z_1| + \dots + |z_n| = z_1 + \dots + z_n$,
	then $z_i = |z_i| \, \forall \, i$.} \label{eqn:z_i}
	\end{equation}
	Assume (a) holds. Then by (a), $|\chi|= 1$, \eqref{eqn:re} 
	and \eqref{eqn:z_i}	we have $m(\rho\chi) \in \RR_{>0}$, whence (b).
	Assume (b),	so that
	$$\sum_{\gamma \in \Gamma} \varphi(\gamma) = \sum_{\gamma \in \Gamma} \varphi(\gamma) \chi(\gamma).$$
	As $|\varphi(\gamma) \chi(\gamma)| = \varphi(\gamma)$ 
	for all $\gamma \in \Gamma$, we deduce that 
	$\varphi(\gamma) \chi(\gamma) = \varphi(\gamma)$ for all $\gamma$,
	which proves (c) as  our functions are frobenian.
	Finally (c) easily implies (b), which obviously implies (a),
	as required.
\end{proof}

\subsection{Frobenian multiplicative functions}
We now introduce the class of multiplicative functions that appear in the statement of Theorem \ref{thm:frob}.
Such multiplicative functions will play a prominent r\^{o}le throughout the paper.

\begin{definition} \label{def:e-weak_frobenian}
	Let  $\eps \in (0,1)$ and let  $\rho: \NN \to \CC$ be a  multiplicative function. 
	We say that $\rho$ is an \emph{$\eps$-weak frobenian multiplicative
	function} if 
	\begin{enumerate}
		\item The restriction of $\rho$ to the set of primes is a frobenian function, in the sense of Definition \ref{def:frob}.
		\item $|\rho(n)| \ll_\varepsilon n^\varepsilon$ for all $n \in \NN$.
		\item There exists $H \in \NN$ such that $|\rho(p^k)| \leq H^k$ for all primes $p$ and all $k \geq 1$.
	\end{enumerate}
	We define the \emph{mean} of $\rho$ to be the mean of the corresponding frobenian function.
\end{definition}


\begin{definition} \label{def:class_F}
 Let  $\rho: \NN \to \CC$ be a  multiplicative function. 
 We say that $\rho$ is a \emph{frobenian multiplicative function} if it is \emph{$\eps$-weak frobenian} for all 
 $\eps \in (0,1)$.
\end{definition}

If $\rho_1$ and $\rho_2$ are ($\eps$-weak) frobenian multiplicative functions, then, by Lemma \ref{lem:product_frob} and Definition \ref{def:e-weak_frobenian}, so is $\rho_1\rho_2$.
In particular $\rho\chi$ is a
frobenian multiplicative function for a Dirichlet character $\chi$ and frobenian multiplicative function $\rho$.

\begin{lemma}\label{lem:sum-rho}
	Let $\eps \in (0,1)$ and $\rho$ be an $\eps$-weak frobenian multiplicative function.
	Then,
	$$\sum_{n \leq x} \rho(n) = c_\rho x (\log x)^{m(\rho)-1}
	 + O(x (\log x)^{m(\rho)-2}) ,$$
    where
    $$
	 c_\rho 
	 = \prod_{p \text{ prime}} 
	   \left(1 + \frac{\rho(p)}{p} + \frac{\rho(p^2)}{p^2} + \dots \right)
	   \left(1-\frac{1}{p}\right)^{m(\rho)}.
	$$
If $\rho$ is real-valued and non-negative with $m(\rho) \neq 0$, then $c_\rho$ is  real and positive.
\end{lemma}
\begin{proof}
    In view of Lemma \ref{lem:primes} (3), this result follows immediately from \cite[Thm.~1.2]{BrTen21}, where $A$ may be taken arbitrarily large, the value of $\rho$ in that statement is given by $m(\rho)$, $r=H$ and $\max(1/2,\eps) < \sigma < 1$. For completeness, we show that the second condition in \cite[(1.11)]{BrTen21} is indeed satisfied, that is 
    \begin{equation} \label{eq:BrTen-assumption}
	\sum_p \bigg\{ \frac{|\rho(p)|^2}{p^{2\sigma}} + \sum_{\nu\geq 2} \frac{|\rho(p^{\nu})|}{p^{\nu \sigma}}\bigg\} < \infty
	\end{equation}
	holds. For any fixed $\nu \geq 2$, part (3) of Definition \ref{def:e-weak_frobenian} implies
	$$
	\sum_{p} |\rho(p^\nu)| p^{-\sigma \nu} \leq H^{\nu} \sum_{p} p^{-2\sigma} < \infty
	$$
	since $\sigma > 1/2$, and the same estimate holds with $|\rho(p^\nu)|$ replaced by $|\rho(p)|^\nu$.
	This estimate holds in particular for $2 \leq \nu \leq 2/ \sigma$. 
	If $\nu > 2/\sigma$, then part (3) of Definition \ref{def:e-weak_frobenian} implies
	$$
	\sum_{p>(2H)^{2/\sigma}} |f(p^\nu)|  p^{-\sigma \nu} 
	\leq \sum_{p>(2H)^{2/\sigma}} 2^{-\nu} p^{-\sigma \nu/2} 
	\ll  2^{-\nu} ((2H)^{2/\sigma})^{-\sigma \nu/2 + 1} \ll (4H)^{-\nu}
	$$
	and $\sum_{\nu > 2/\sigma} (4H)^{-\nu} < \infty$.
	For the remaining sum over small primes, part (2) implies that
	$$\sum_{p\leq H^{2/\sigma}} \sum_{\nu \geq 2} |\rho(p^\nu)| p^{-\sigma \nu} 
	\ll \sum_{p\leq H^{2/\sigma}} \sum_{\nu \geq 2} p^{(-\sigma+\eps)\nu} < \infty ,$$
	since $\sigma > \eps$.
	The fact that $c_{\rho}$ converges as well as the final part of our lemma follow from \eqref{eq:BrTen-assumption} and Lemma \ref{lem:primes}.
	
	We observe that under the additional assumption that $|\rho| \leq \tau_H$, where $\tau_H$ denotes the multiplicative function with Dirichlet series $\zeta^{H}(s)$, the conclusion of the lemma would follow from \cite[Thm.~1]{GraKou19}. This would be sufficient for all later applications to $\{0,1\}$-valued frobenian multiplicative functions.
\end{proof}

\section{Frobenian multiplicative functions evaluated at linear polynomials} \label{sec:Lilian}
In this section we prove Theorem \ref{thm:frob}. 
The main technical tool upon which our proof relies is a special case of the main result from \cite{Mat16}, 
namely \cite[Theorem~2.1]{Mat16}. 
The first two subsections below contain the preparation for applying this tool.
In the first subsection, we describe a general class of multiplicative functions and verify that
frobenian multiplicative functions belong to that class.
In the second subsection, we deduce a version of the relevant result from \cite{Mat16} 
that is adjusted to our situation.
Finally, the third subsection contains the proof of Theorem \ref{thm:frob}.

\subsection{Set-up and verification}
Given any arithmetic function $h: \NN \to \CC$, $x\geq 1$ and $q,A \in \ZZ$, $q \not=0$, we define
$$
S_h(x)= 
\frac{1}{x} 
\sum_{\substack{1\leq n \leq x}} 
h(n), \qquad
S_h(x;q,A)= 
\frac{q}{x} 
\sum_{\substack{1\leq n \leq x \\ n\equiv A \Mod{q} }} 
h(n)$$
to be the average value of $h$ up to $x$ and the average value of $h$ in the progression $A \Mod{q}$ up to $x$, respectively.
Moreover, for $x>1$ and $q \in \NN$, let
\begin{align}
 E_{h}(x;q) &= \frac{1}{\log x} \frac{q}{\phi( q )}
\prod_{p \leq x, p \nmid q} \left(1 + \frac{|h(p)|}{p} \right).
\end{align}

The results from \cite{Mat16} apply to a general class $\mathcal{F}^*$ of multiplicative functions which contains the
following class $\mathcal{F}$ as a subset.
\begin{definition} \label{def:F}
Let $\mathcal{F}$ denote the class of multiplicative functions
$h: \NN \to \CC$ with the properties: 
\begin{itemize}
 \item[(i)] There exists a constant $H \in \NN$, depending on $h$, such that 
$|h(p^k)| \leq H^k$ for all primes $p$ and all integers $k \geq 1$;
 \item[(ii)] $|h(n)| \ll_{\eps,h} n^{\eps}$ for all $n \in \NN$ and all $\eps > 0$; 
 \item[(iii)] There exists a positive constant $\alpha_h$ such that 
$$
\frac{1}{x}
\sum_{p \leq x} |h(p)| \log p \geq \alpha_h 
$$
for all sufficiently large $x$; and
\item[(iv)] $h$ \emph{has a stable mean value in arithmetic progressions, i.e.:}
For every constant $C>0$, there exists a function 
$\psi_C$ with $\psi_C(x) \to 0$ as $x\to \infty$ such that the estimate
$$
S_{h}(x';q,A)
= S_{h}(x;q,A)
 + O\Big(\psi_C(x) E_{h}(x;q)\Big)
$$
holds for all $x \geq 2$ and $x' \in (x(\log x)^{-C},x)$,
and for all progressions $A \Mod{q}$ with $\gcd(q,A)=1$, 
where $1 < q \leq (\log x)^C$ and $p \mid q$ for every prime $p < \log \log x$.
\end{itemize}
\end{definition}

We note as an aside that if $h:\NN \to \CC$ is multiplicative and satisfies the conditions 
(i) and (ii) from above, then a special case of Shiu \cite[Theorem 1]{Shiu80} implies that, as $x \to \infty$,
we have $|S_h(x,q,A)| \leq S_{|h|}(x,q,A) \ll E_h(x;q)$
uniformly for all $q < x^{3/4}$ and $0<A<q$ such that $\gcd(A,q)=1$.

It is often easier to work with bounded multiplicative functions than with functions from the 
general class $\mathcal{F}$.
Similarly, working with completely multiplicative functions will often be easier than working with 
general multiplicative functions.
To handle the general case in our setting, we will make use of the following two tools that allow
us to reduce our case to either of the two easier settings:

In the setting of Definition \ref{def:F}, `bounded' corresponds to the case where we may take $H=1$. 
In order to invoke, even when $H>1$, results that a priori only apply to bounded multiplicative functions, 
we follow \cite{Mat14} and associate to any given $h \in \mathcal{F}$ with $H>1$ the bounded multiplicative 
function $g_h: \NN \to \CC$ whose values at prime powers are given by:
\begin{equation} \label{eq:def-g}
 g_h(p^k) =
 \begin{cases}
 h(p)/H, & \text{if } k=1,\cr
 0,& \text{if } k>1.
 \end{cases}
\end{equation}
If $H=1$, we set $g_h=h$.
The function $g_h$ is defined in such a way that $h$ can be decomposed as the 
convolution $h=g_h^{(*H)} * g'_h$ of $H$ copies of the bounded function $g_h$ and 
one copy of a function $g'_h$ that is (away from $1$) supported on square-full numbers only.
Observe that if $h$ is frobenian, then so is $g_h$.

The second tool addresses the problem that sieving becomes difficult when the 
function at hand is not completely multiplicative.
Recall the notion of an $\eps$-weak frobenian function from Definition \ref{def:e-weak_frobenian}.
\begin{lemma} \label{lem:tilde-rho}
 Let $\rho$ be a frobenian multiplicative function, let
 $H \in \NN$ be such that Definition \ref{def:e-weak_frobenian}(3) holds,
 let $q$ be a positive integer and let $\eps \in (0,1)$. 
 If $\tilde \rho$ is the multiplicative function defined via
 \begin{equation} \label{eq:def-tilde-rho}
 \tilde \rho (p^k) =
 \begin{cases}
 0, &\text{if } p\mid q \text{ and } p \leq H^{1/\eps},\cr
 \rho(p)^k, & \text{if } p\mid q \text{ and } p > H^{1/\eps},\cr
 \rho(p^k), &\text{if } p\nmid q,
 \end{cases}
\end{equation}
 then $\tilde \rho$ is $\eps$-weak frobenian and $m(\rho)=m(\tilde \rho)$.
\end{lemma}

\begin{proof} 
 Let $\varphi$ be a class function for which $\rho(p) = \varphi(\Frob_p)$ for all primes outside some set of places $S$. Enlarging $S$ if necessary to include all primes $p \leq H^{1/\eps}$, 
 part (1) of Definition \ref{def:e-weak_frobenian} holds for $\tilde \rho$ with the same class function $\varphi$. 
 Next, note that $|\tilde\rho(p^k)| \leq H^k \leq p^{\eps k}$ holds for all $p \mid q$ by construction. 
 For $p \nmid q$ we have $|\tilde\rho(p^k)| \ll p^{\eps k}$ since $\rho$ is frobenian. 
 Hence $\tilde\rho$ is $\eps$-weak frobenian.
  Since $m(\rho)$ only depends on the class function $\varphi$ from Definition \ref{def:frob}, 
 we have $m(\tilde\rho)=m(\rho)$.
\end{proof}

The rest of this subsection is devoted to proving the following proposition.

\begin{proposition} \label{lem:verification}
 If $\rho$ is a real-valued non-negative frobenian multiplicative function with $m(\rho) > 0$, then
 $\rho \in \mathcal{F}$.
\end{proposition}

Conditions $(i)$ and $(ii)$ of Definition \ref{def:F} are immediate as they are part of the defining 
properties of frobenian multiplicative functions. 
Condition $(iii)$ holds for any
$\alpha_{\rho} = m(\rho) - \delta$ with $0 < \delta < m(\rho)$, as follows from Lemma \ref{lem:primes}(3) and the assumption that $m(\rho) > 0$.
The main difficulty, thus, lies in establishing condition~(iv), and we begin 
by analysing the relevant mean values of 
frobenian multiplicative functions in progressions.
\begin{lemma} \label{lem:rho-in-APs}
Let $\rho: \NN \to \RR_{\geq 0}$ be a non-negative frobenian multiplicative function, let 
$\mathcal{E}_{\rho}$ denote the (finite) set of primitive Dirichlet characters $\chi$ for which
$m(\rho \chi) = m(\rho)$, and let $\mathcal{E}_{\rho}(q)$ denote the set of characters modulo $q$ that are induced 
by the primitive characters $\chi \in \mathcal{E}_{\rho}$. Let $C > 1$ be fixed and 
$x>e$ be a parameter. Then,
\begin{align*}
\Bigg| S_{\rho}(X; q,A) -
\frac{q}{\phi(q)} 
\sum_{\chi \in \mathcal{E}_{\rho}(q)} \chi(A) 
\frac{1}{X} \sum_{n \leq X} \rho(n) \bar\chi(n) \Bigg|
= o_{x \to \infty}(1)E_{\rho}(x;q),
\end{align*}
uniformly for all 
$x^{1/2} \leq X \leq x$, all moduli $q \leq (\log x)^C$ such that $p \mid q$ for all primes $p \leq \log \log x$, 
and all $A \in (\ZZ/q\ZZ)^*$.
\end{lemma}

\begin{proof}
Recall that that $\mathcal{E}_\rho$ is a finite set by Lemma \ref{lem:finitely-many}. 
We seek to apply \cite[Cor.\ 4.2]{Mat14} which is an easy corollary to a result of
Granville and Soundararajan \cite{GS-book} but requires some set-up.
For this purpose, let $g_{\rho}$ denote the bounded multiplicative function obtained via  
\eqref{eq:def-g} for $h=\rho$, and let $H \in \NN$ be as in Definition \ref{def:F}(i) for $\rho$.
Suppose we are given any $x > e$, let $y \in [x^{1/(8H)},x]$ and 
enumerate for each fixed $y$ the primitive characters of conductor at most $(\log y)^C$ as 
$\chi_1, \chi_2, \dots$ in such a way that 
$|S_{g_{\rho} \bar\chi_1}(y)| \geq |S_{g_{\rho} \bar\chi_2}(y)| \geq \dots$
forms a non-decreasing sequence. 
Note that $m(g_{\rho}\chi) = \frac{1}{H}m(\rho\chi)$ for any character $\chi$. 
Since $\rho$ is real-valued and non-negative, Lemmas \ref{lem:finitely-many} and~\ref{lem:sum-rho} imply that
for all sufficiently large $x$ and for each choice of $y$ as above, the initial elements of the associated sequences 
$\chi_1, \chi_2, \dots$ are precisely given by those $\chi \in \mathcal{E}_{\rho}$ for which 
$c_{g_{\rho} \bar \chi} \not=0$ in Lemma \ref{lem:sum-rho}.
Let $\alpha_{\rho} > 0$ be as in Definition \ref{def:F}(iii) for $\rho$ and 
let $k \geq \max(2, \alpha_{\rho}^{-2}, \#\mathcal{E}_{\rho})$ be an integer. 
For any $y$ as before, define the set $\mathcal{E}_{\rho, k, y} = \{ \chi_1, \dots, \chi_k\}$ to consist of the 
first $k$ elements of the sequence of characters defined for the given value of $y$.
Moreover, let $\mathcal{E}_{\rho, k} = \bigcup_{1\leq j \leq z} \mathcal{E}_{\rho, k, x^{1/2^{j}}}$, where
$z= \lceil \log_2(4H)\rceil$,
and let $\mathcal{E}_{\rho, k}^*(q)$ denote the set of characters $\chi^* \Mod{q}$ that are induced from
the characters $\chi \in  \mathcal{E}_{\rho, k}$.
Then \cite[Cor.~4.2]{Mat14} implies that 
\begin{align*}
\bigg| S_{\rho}(X; q,A) -
\frac{q}{\phi(q)} 
\sum_{\chi^* \in \mathcal{E}^*} \chi^*(A) 
\frac{1}{X}
\sum_{n \leq X} \rho(n) \bar\chi^*(n) \bigg|
= o_{x \to \infty}(1) E_{\rho}(x;q),
\end{align*}
uniformly for all $x^{1/2} \leq X \leq x$, all $q \leq (\log x)^C$ as in the statement, all $A \in (\ZZ/q\ZZ)^*$,
and for all sets $\mathcal{E}^* \supseteq \mathcal{E}_{\rho, k}^*(q)$ of Dirichlet characters modulo $q$. 
The lemma thus follows provided we can show that 
\begin{equation} \label{eq:aux-bd-3.4}
S_{\rho \bar\chi^*}(X) 
= o_{x \to \infty}
\bigg(
\frac{1}{\log x} \prod_{p \leq x, \, p\nmid q}\left(1+ \frac{|\rho(p)|}{p}\right)
\bigg) 
\end{equation} 
for every $\chi^* \in  \mathcal{E}^*_{\rho, k}(q)$ 
that is induced from some $\chi \in \mathcal{E}_{\rho, k} \setminus \mathcal{E}_{\rho}$,
and where $q \leq (\log x)^C$ is such that $p\mid q$ for all $p\leq \log \log x$.

To prove \eqref{eq:aux-bd-3.4}, let $\rho_q = \tilde \rho$ denote the function defined by \eqref{eq:def-tilde-rho} 
for the given value of $q$ and for $\eps = 1/2$, say.
Then
$S_{\rho \bar\chi^*}(X)
= S_{\rho_q \bar\chi^*}(X)$,
and Lemma \ref{lem:tilde-rho} implies that $m(\rho) = m(\rho_q)$ as well as  
$m(\rho_q \chi) = m(\rho \chi) = H m(g_{\rho} \chi)$ for all characters $\chi$.
To prove the required bound, note that for each 
$\chi \in \mathcal{E}_{\rho, k} \setminus \mathcal{E}_{\rho}$ 
there is some $\delta > 0$ such that $\re m(\rho \chi) < m(\rho) - \delta$. 
By Lemma \ref{lem:finitely-many}, we in fact have $m(\chi\rho) = 0$ for all but finitely many primitive $\chi$.
Combining this information with Lemma \ref{lem:sum-rho},
there thus exists $\delta_0 >0$ such that 
$$
S_{\rho_q \bar\chi}(X) 
\ll (\log x)^{m(\rho)-1-\delta_0},
$$
uniformly for all $x^{1/2}(\log x)^{-C} \leq X \leq x$ and $\chi$ as before.
Since $\rho_q$ is completely multiplicative at primes dividing $q$, M\"obius inversion and the above 
yield
\begin{align*}
S_{\rho_q \bar\chi^*}(X)
&= \sum_{d|q} \frac{ \mu(d) \rho(d) \bar\chi (d)}{d} S_{\rho_q \bar\chi}(X/d) 
\ll \sum_{d|q} \frac{ |\rho(d)|}{d} |S_{\rho_q \bar\chi}(X/d)| \\
&\ll (\log x)^{m(\rho)-1-\delta_0} \prod_{p\mid q} \left(1 + \frac{|\rho(p)|}{p} \right),
\end{align*}
provided $x^{1/2} \leq X \leq x$ and $q \leq (\log x)^C$.
Invoking the final part of Lemma~\ref{lem:primes}, the bound $q \leq (\log x)^C$
and $|\rho(p)| \leq H$, we deduce that
\begin{align*}
S_{\rho_q \bar\chi^*}(X)
&\ll (\log x)^{-1-\delta_0} \prod_{p\mid q} \left(1 + \frac{|\rho(p)|}{p} \right) 
\prod_{p' \leq x} \left(1 + \frac{|\rho(p')|}{p'} \right)\\
&\ll (\log x)^{-1-\delta_0} \prod_{p \mid q} \left(1 + \frac{|\rho(p)|}{p} \right)^2 
\prod_{p' \leq x, \, p'\nmid q} \left(1 + \frac{|\rho(p')|}{p'} \right)\\
&\ll (\log q)^{2H} (\log x)^{-1-\delta_0} \prod_{p' \leq x, \, p'\nmid q} \left(1 + \frac{|\rho(p')|}{p'} \right)\\
&\ll (\log x)^{-1-\delta_0/2}\prod_{p \leq x, \, p\nmid q} \left(1 + \frac{|\rho(p)|}{p} \right).
\end{align*}
Hence \eqref{eq:aux-bd-3.4} holds as required.
\end{proof}

We will apply Lemma \ref{lem:rho-in-APs} together with the following refinement of Lemma~\ref{lem:sum-rho}.

\begin{lemma}\label{lem:sum-rho-chi}
Let $\rho$ be a frobenian multiplicative function, 
let $H$ be as in Definition \ref{def:e-weak_frobenian}(3), 
let $\mathcal{S}_0 = \{p \text{ prime}: p \leq H^{8}\}$, and let $x>e$ be a parameter.

If $\rho^*$ denotes the completely multiplicative function whose values at primes
are given by $\rho^*(p)=\rho(p)\1_{p \not\in \mathcal{S}_0}$, and if 
$c_{\rho^*}$ denotes the corresponding leading constant from Lemma \ref{lem:sum-rho},
then, as $x \to \infty$, 
\begin{align*}
\frac{1}{x} \sum_{\substack{n \leq x \\ \gcd(n,q)=1}} \rho(n)  
&=  c_{\rho^*} (\log x)^{m(\rho)-1} 
   \prod_{\substack{p\mid q,\, p \not\in \mathcal{S}_0}} \left(1 - \frac{\rho(p)}{p} \right)
   + o_{x \to \infty}\left( \frac{\phi(q)E_{\rho}(x,q)}{q} \right)
\end{align*}
uniformly for all integers $q \leq \exp((\log \log x)^2)$ such that $p\mid q$ for all 
primes $p < \log \log x$.
Moreover, if $\rho$ is real non-negative with $m(\rho)>0$, then $c_{\rho^*}>0$.
\end{lemma}

\begin{proof}
Let $q \leq \exp((\log \log x)^2)$ be such that $p\mid q$ for all primes $p < \log \log x$, 
and note that $m(\rho) = m(\rho^*)$. 
To start with, we claim that
\begin{equation} \label{eq:rho*_claim}
   \sum_{\substack{n \leq x \\ \gcd(n,q)=1}} | \rho(n) - \rho^*(n)|
   = o_{x \to \infty} \Bigg(\frac{1}{\log x}
     \prod_{p \leq x, p \nmid q} \left(1 + \frac{|\rho(p)|}{p}\right)\Bigg). 
\end{equation}
Assuming \eqref{eq:rho*_claim} for the moment, it suffices to prove the lemma with $\rho$ replaced 
by $\rho^*$. 
In this case, it follows from Lemma \ref{lem:sum-rho} and M\"{o}bius inversion that
\begin{align*}
   \sum_{\substack{n \leq x \\ \gcd(n,q)=1}} \rho^*(n)
&= \sum_{d\mid q} \mu(d) \rho^*(d) 
   \sum_{n \leq x/d} \rho^*(n) \\  
&= \sum_{d\mid q} \frac{\mu(d) \rho^*(d)}{d} 
   \Big(c_{\rho^*} + O\Big(\frac{1}{\log (x/d)}\Big)\Big) 
   x(\log (x/d))^{m(\rho)-1} \\
&= \left(c_{\rho^*} + O_{\delta}\left((\log x)^{-1 + \delta}\right)\right) 
   x (\log x)^{m(\rho)-1} \prod_{p\mid q, p\not\in \mathcal{S}_0} \left(1 - \frac{\rho(p)}{p} \right),   
\end{align*}
where we used the bounds $\log (x/d) = (1+O((\log \log x)^2/\log x)) \log x$ for $d \mid q$,
and $\log \log x \ll_{\delta} (\log x)^{\delta}$ for $\delta >0$, as well as
$$
\sum_{d\mid q} \frac{|\mu(d) \rho^*(d)|}{d}
= \prod_{p\mid q, p\not\in \mathcal{S}_0} \left(1 + \frac{|\rho(p)|}{p} \right)
\leq \prod_{p\mid q} \left(1 + \frac{H}{p} \right)
\ll (\log \log x)^{2H}.
$$

Thus, it remains to prove \eqref{eq:rho*_claim}.
Note that $|\rho(n)|,|\rho^*(n)| \ll n^{1/8}$ for all $n \in \NN$,
and that $\rho(n)=\rho^*(n)$ for all square-free integers $n$ that are co-prime to $q$.
Let us decompose each integer $n = m_1 m_2$ into a product of a square-free integer
$m_1$ and a square $m_2=m^2$.
Then, assuming that $\gcd(n,q)=1$, 
the condition $\rho(n) \neq \rho^*(n)$ implies that $m$ has a prime factor $p\geq \log \log x$.
We thus have:
\begin{align} \label{eq:another-bound}
&\frac{1}{x} \Big|
\sum_{\substack{n \leq x \\ \gcd(n,q)=1}} \left(\rho(n) - \rho^*(n)\right) \Big|  \ll
\\
\nonumber
& \frac{1}{x}
\sum_{\substack{m_1 \leq x^{1/2}: \\ |\mu(m_1)|=1}} |\rho^*(m_1)| 
\sum_{\substack{1 < m^2 \leq x/m_1 :\\ \gcd(m,q)=1 }} m^{1/4} 
+
\sum_{\substack{1<m^2 \leq x^{1/2}: \\ \gcd(m,q)=1 }} 
\frac{m^{1/4}}{m^2}
\frac{m^2}{x}
\sum_{\substack{m_1 \leq x/m^2 \\ \gcd(m_1,q)=1}}  
|\rho^*(m_1)|.
\end{align}
Using the bound $|\rho^*(m_1)|\leq H^{\Omega(m_1)}$,
the first of the two terms is bounded by:
\begin{align*} 
 &\ll \sum_{\substack{m_1 \leq x^{1/2}\\ |\mu(m_1)|=1}} \frac{H^{\omega(m_1)}}{m_1} 
 \frac{m_1}{x}\sum_{\substack{m \leq (x/m_1)^{1/2} }} m^{1/4}
 \ll  \sum_{\substack{m_1 \leq x^{1/2}\\ |\mu(m_1)|=1 }} \frac{H^{\omega(m_1)}}{m_1} (x/m_1)^{-1/2 + 1/4}\\ 
 &\ll  x^{-1/8} \prod_{p\leq x^{1/2}}(1 + H/p) 
 \ll_{\eps}  x^{-1/8 +\eps},
\end{align*}
which agrees with our claim.
Concerning the second term in the bound \eqref{eq:another-bound}, it follows from
Shiu \cite[Theorem 1]{Shiu80} (see \cite[Lemma 3.1]{Mat14}) that the inner sum satisfies
$$
\frac{m^2}{x}
\sum_{\substack{m_1 \leq x/m^2 \\ \gcd(m_1,q)=1}}  
|\rho^*(m_1)| 
\ll \frac{1}{\log x} 
\prod_{p \leq x, p \nmid q} \left(1 + \frac{|\rho(p)|}{p}\right).
$$
For the outer sum, we have
\begin{align*}
& \sum_{\substack{1<m^2 \leq x^{1/2} \\\exists p \geq (\log \log x). p \mid m }} 
\frac{m^{1/4}}{m^2}
\leq \sum_{\substack{(\log \log x)^2 < m^2 \leq x^{1/2} }} 
m^{-2+1/4}
\ll (\log \log x)^{-1+1/4},
\end{align*}
which shows that the second term, too, is $o(\phi(q) E_{\rho}(x;q)/q)$, as required.
\end{proof}

We are now in the position to verify condition (iv) of Definition~\ref{def:F} for real
non-negative frobenian multiplicative functions.

\begin{lemma} \label{lem:condition-iv}
 Let $\rho: \NN \to \RR_{\geq 0}$ be a real non-negative frobenian multiplicative function.
Then, with all assumptions from Definition~\ref{def:F}(iv) in place, we have 
$$S_{\rho}(x; q, A) = S_{\rho}(x'; q, A) + o_{x \to \infty}(1) E_{\rho}(x;q).$$ 
\end{lemma}
\begin{proof}
 Lemma \ref{lem:rho-in-APs} yields an approximation of $S_{\rho}(X; q, A)$ 
 by a finite character sum that holds uniformly for all $X \in [x^{1/2},x]$.
 Using this approximation, the lemma follows provided
$$
S_{\rho \chi^*}(x) = S_{\rho \chi^*}(x') 
+ o_{x \to \infty}\bigg(\frac{1}{\log x} \prod_{p \leq x}\bigg(1+ \frac{|\rho(p)\chi^*(p)|}{p}\bigg)\bigg)
$$
for all $\chi^* \in \mathcal{E}_{\rho}(q)$. 
If $\chi$ denotes the primitive character that induces $\chi^*$, then the latter assertion follows from 
Lemma \ref{lem:sum-rho-chi} applied with $\rho$ replaced by $\rho \chi$.
Indeed, since $\log x' = \log x + O(C\log \log x)$, applying the lemma to both terms in the difference
$S_{\rho \chi^*}(x) - S_{\rho \chi^*}(x')$, we obtain sufficient cancellation in main terms.
\end{proof}

\begin{proof}[Proof of Proposition \ref{lem:verification}]
 Conditions $(i)$ and $(ii)$ are clear, 
 condition $(iii)$ follows with $\alpha_{\rho} = \frac{m(\rho)}{2}$ from Lemma \ref{lem:primes}(3), 
 while condition $(iv)$ holds by Lemma~\ref{lem:condition-iv}.
\end{proof}

\subsection{Correlations of frobenian multiplicative functions}

In this section, we deduce an asymptotic result for correlations of frobenian multiplicative functions 
from \cite[Theorem 2.1]{Mat16}. 
In view of Proposition \ref{lem:verification}, we could apply \cite[Theorem 2.1]{Mat16} directly.
However, in the case of frobenian multiplicative functions a stronger result can in fact be obtained.

\begin{definition}
For any real number $x>e$, define $$W(x) := \prod_{p < \log \log x} p.$$
\end{definition}

\begin{definition} \label{def:W-tilde}
Given any fixed collection $\rho_1, \dots, \rho_r$ of frobenian multiplicative functions, 
we define the following function $\widetilde{W}(x)=\widetilde{W}(x;\rho_1, \dots, \rho_r)$.
For each $1\leq j \leq r$, let $\mathcal{E}_{\rho_j}$ denote the set of primitive characters 
defined in Lemma \ref{lem:rho-in-APs}.
If $q_{\chi}$ denotes the conductor of the character $\chi$, define
$$\widetilde W(x) 
 = W(x) \prod_{j=1}^r \prod_{\chi \in \mathcal{E}_{\rho_j}} q_{\chi}~, \quad (x>e).
$$
\end{definition}
Regarding $r$ and $\rho_1, \dots, \rho_r$ as fixed, we have
$\widetilde W(x) \ll (\log x)^{1+o(1)}$.
 
\begin{definition}[Finite complexity system of linear polynomials]
 Let $\varphi_1, \dots, \varphi_r$ $\in \ZZ[u_1, \dots, u_s]$ be linear polynomials.
 Then ${\boldsymbol \varphi} = (\varphi_1, \dots, \varphi_r)$ is called a 
 finite complexity system of linear polynomials if for any 
 pair of indices $i \not= j$, the linear forms 
 $\psi_i(\boldsymbol u) := \varphi_i(\boldsymbol u) - \varphi_i(\boldsymbol 0)$ and
 $\psi_j(\boldsymbol u) := \varphi_j(\boldsymbol u) - \varphi_j(\boldsymbol 0)$ are linearly independent over $\QQ$.
\end{definition}

Restricted to the class of frobenian multiplicative functions, \cite[Theorem 2.1]{Mat16} yields the following:
\begin{theorem} \label{thm:mat16}
Let $N>2$ be an integer parameter,
let $\rho_1, \dots, \rho_r: \NN \to \RR_{\geq 0}$ be non-negative frobenian multiplicative functions, each satisfying 
$m(\rho_j) > 0$, and let $\widetilde{W}=\widetilde{W}(N)$ be as in Definition~\ref{def:W-tilde}.
Further, let $\varphi_1, \dots, \varphi_r \in \ZZ[u_1, \dots, u_s]$ be a finite complexity system of linear polynomials,
let $\mathfrak{K} \subset [-1,1]^s$ be a fixed convex set and let
$$\mathfrak{K}^+ = \mathfrak{K} \cap \bigcap_{j=1}^r \psi_j^{-1}(\RR^+)$$
be the (convex) subset of $\mathfrak{K}$ that is mapped to $\RR^+$ by each of the linear forms
$\psi_j(\boldsymbol u) := \varphi_j(\boldsymbol u) - \varphi_j(\boldsymbol 0)$.
Finally, suppose that $\vol(\mathfrak{K}^+)>0$, extend each $\rho_j$ to all of $\ZZ$ by setting $\rho_j(-m) = 0$ if $m \geq 0$, and fix a point $\boldsymbol{a} \in \RR^s$.

Then there exists a positive constant $B_2$ such that the following asymptotic holds as $N \to \infty$:
 \begin{align}\label{eq:main'}
\nonumber
&\frac{1}{\vol (N\fK^+)}
\sum_{\substack{\boldsymbol n \in \ZZ^s \cap (N\fK + \boldsymbol{a})}} \prod_{j=1}^r 
\rho_j(\varphi_j(\boldsymbol n))
~= \\
\nonumber
&\sum_{\substack{w_1, \dots, w_r 
 \\ p\mid w_i \Rightarrow p\mid \widetilde W
 \\ w_i \leq (\log N)^{B_2}}}
\sum_{\substack{A_1,\dots,A_r \\ \in (\ZZ/\widetilde W \ZZ)^*}}
\bigg(\prod_{j=1}^r 
\rho_j(w_j)S_{\rho_j}\Big(N;\widetilde W,A_j\Big)\bigg)
\beta_{\boldsymbol{\varphi}}(w_1A_1, \dots, w_rA_r)\\
&\qquad \qquad  +
o\bigg( \frac{1}{(\log N)^r} 
\prod_{j=1}^r  \prod_{p \leq N} 
\bigg(1 + \frac{\rho_i(p)}{p} \bigg)
\bigg)
~,
\end{align}
where
\begin{equation*}
\beta_{\boldsymbol{\varphi}}(w_1A_1, \dots, w_rA_r)
=\frac{1}{(w\widetilde W)^s}
\sum_{\substack{{\boldsymbol v} \in \\ (\ZZ/w \widetilde W \ZZ)^s}}
\prod_{j=1}^r 
\1_{\varphi_j({\boldsymbol v}) \equiv w_j A_j ~(w_j \widetilde W)}
\end{equation*}
with $w = \lcm(w_1, \dots, w_r)$.
\end{theorem}

\begin{proof}
 Our first aim is to show that we can replace the set $N\fK + \boldsymbol{a}$ in the summation condition
 on the left hand side by $N\fK^+$.
 To this end, we start by showing that this change only involves changing the summation domain on a set
 of volume $O(N^{s-1})$, if we ignore points in the domain at which the summation argument is zero.
 Recall the notation 
 $\mathcal{A} \Delta \mathcal{B}=(\mathcal{A}\cup\mathcal{B})\setminus(\mathcal{A}\cap\mathcal{B})$.
 Replacing $N\fK + \boldsymbol{a}$ by $N\fK$ changes the summation domain by the set 
 $(N\fK + \boldsymbol{a})\Delta (N\fK)$, which is contained in the $\|\boldsymbol{a}\|$-neighbourhood 
 of the boundary of $N\fK$.
 Since $\boldsymbol{a}$ is fixed and $\fK$ convex, this $\|\boldsymbol{a}\|$-neighbourhood has a volume
 of order $O(N^{s-1})$, see e.g.\ \cite[Cor.\ A.2]{GT-linearprimes}, and thus
 $$\vol\Big((N\fK + \boldsymbol{a})\Delta (N\fK)\Big) = O(N^{s-1}).$$
 
 Since all $\rho_j$ vanish on $\ZZ_{\leq 0}$, the domain $N\fK$ can immediately be replaced by
 $$
 \fK_N^+:=(N\fK) \cap \bigcap_{j=1}^n \varphi_j^{-1}(\RR^+) 
 = \{\boldsymbol{n} \in N\fK: 
 \varphi_j (\boldsymbol{n}) = \psi_j (\boldsymbol{n}) + \varphi_j(\boldsymbol{0}) > 0 
 \text{ for all }j 
 \}.
 $$
 In order to compare this set to the set $\fK^+$ from the statement, we note that,
 since each of the $\psi_j$ is homogeneous, we have
 $$
 N\fK^+
 = \{\boldsymbol{n} \in N\fK: \psi_j (\boldsymbol{n}) > 0 \text{ for all }j\}.
 $$
 Writing $b_j=\varphi_j(\boldsymbol{0})$, it thus follows that
 $$
 \vol(\fK_N^+ \Delta N\fK^+)
 \leq \sum_{j=1}^r\vol(\{ \boldsymbol{n} \in N\fK : \psi_j (\boldsymbol{n}) \in [-b_j,b_j]\})
 =O(N^{s-1}).
 $$
 
 The above information will be used to bound one factor in an application of Cauchy-Schwarz,
 while the second factor will be handled with the help of the following bound.
 Let $\rho$ denote the multiplicative function whose values at prime powers are
 given by $\rho(p^j) = \max(|\rho_1(p^j)|, \dots, |\rho_r(p^j)|)$.
 If $H \in \NN$ is such that Definition \ref{def:e-weak_frobenian}(3) holds for all the $\rho_j$, then
 \cite[Lemma 7.9]{BM17} implies that 
 $$
 \sum_{\boldsymbol{n} \in \ZZ^s \cap N\fK^* } \prod_{i=1}^r
 \rho(\varphi_j(\boldsymbol{n}))^2 \ll_C  |\ZZ^s\cap N\fK^*| (\log N)^{O_{r,H}(1)}
 $$
 for any bounded convex subset $\fK^* \subseteq [-C,C]^s$.
 To use this bound, let $\fK_{\boldsymbol{a}}:= 
 \{\boldsymbol{k} + \lambda \boldsymbol{a} \mid \boldsymbol{k} \in \fK, \lambda \in [0,1]\}$
 and note that $N\fK \cup( N\fK + \boldsymbol{a}) \subseteq N \fK_{\boldsymbol{a}}$. 
 The error incurred by replacing $N\fK + a$ by $N\fK^+$ on the left hand side of \eqref{eq:main'}
 can be bounded by
 \begin{align*}
& \frac{1}{\vol N\fK^+}
 \sum_{\substack{\boldsymbol n \in \ZZ^s \cap ((N\fK + a)\Delta N\fK)\cup(N\fK^+ \Delta \fK^+_N)) }} 
 \prod_{j=1}^r 
 |\rho_j(\varphi_j(\boldsymbol n))| \\
& \ll
 \frac{1}{N^s}
 \sum_{\substack{\boldsymbol n \in \ZZ^s \cap N \fK_{\boldsymbol{a}}}} \prod_{i=1}^r 
 \rho(\varphi_i(\boldsymbol n)) \1_{\boldsymbol n \in ((N\fK + a)\Delta N\fK)\cup(N\fK^+ \Delta \fK^+_N)}.
 \end{align*}
 Applying Cauchy-Schwarz to the latter expression and invoking the above second moment bound 
 as well as the bounds on the volumes of the sets in the indicator function, our error term is seen 
 to be $O(N^{-1}(\log N)^{O_{r,H}(1)})$, which is negligible in view of the error term in
 \eqref{eq:main'}.
 
 Replacing thus $N\fK + \boldsymbol{a}$ by $N\fK^+$ on the left hand side of \eqref{eq:main'} and in view of Proposition \ref{lem:verification}, we are left with 
 an expression to which \cite[Theorem 2.1]{Mat16} can be applied.
 In view of the error term in \eqref{eq:main'}, the conclusion of  Theorem \ref{thm:mat16} is, 
 however, stronger than what is implied by a direct application of \cite[Theorem 2.1]{Mat16}.
 The reason behind this is that in the special case of frobenian multiplicative functions, 
 we can prove a stronger `$W$-trick'.
 More precisely, Lemma \ref{lem:rho-in-APs} shows that the set of primitive characters that determine the behaviour
 of the mean value of a frobenian multiplicative function $\rho$ in progressions is a fixed set that does not depend 
 on the cut-off parameter $x$ as soon as $x$ is sufficiently large.
 In the general setting of \cite{Mat16} one has to, instead of with this fixed set, work with the set $\mathcal{E}_{\rho,k}$ that appeared
 in the proof of Lemma \ref{lem:rho-in-APs} and might depend on $x$ and $C$.
 Running through the proof of \cite[Proposition 5.1]{Mat14} with $\mathcal{E}$ replaced by our fixed set 
 $\mathcal{E}_{\rho}$, we see that if $q_{\chi}$ denotes the conductor of a character $\chi$, then
 $\widetilde W(x) 
 = W(x) \prod_{j=1}^r \prod_{\chi \in \mathcal{E}_{\rho_j}} q_{\chi}$
 satisfies the conclusion of \cite[Proposition 5.1]{Mat14}.
 In particular, the value of $\kappa$ in \cite[Proposition 5.1]{Mat14} can be chosen independent of $E$ in this case,
 and it follows, moreover, that the $W$-trick in \cite[Theorem 6.1]{Mat14} is independent of the degree and dimension of the 
 nilsequence involved.
 This in turn allows us to take limits in the application of the inverse theorem in the proof of
 \cite[Theorem 2.1]{Mat16} (more precisely in the proof of the auxiliary result stated in \cite[Proposition 4.2]{Mat16}) without changing the $W$-trick. 
 Hence, $\eps$ can be omitted from the conclusion of \cite[Theorem 2.1]{Mat16} when applied to 
 frobenian multiplicative functions.
\end{proof}

\subsection{Proof of Theorem \ref{thm:frob}}

Theorem \ref{thm:mat16} applies in the situation of Theorem \ref{thm:frob} 
with $s = n+1$ and $N = B$. 
As an intermediate step, we will prove:

\begin{proposition} \label{prop:right-order-asymptotics}
 Let $\mathfrak{K} \subset [-1,1]^n$, $\rho_1, \dots, \rho_r: \NN \to \RR_{\geq 0}$ and 
 $L_1, \dots, L_r \in \ZZ[x_0, \dots,x_n]$ be as in Theorem \ref{thm:frob}. 
 If $s = n+1$, $N = B$ and $\varphi_i(\xx)=L_i(\xx)$ for $i=1, \dots, r$,
 and provided $\vol \fK^+>0$ in the notation of Theorem \ref{thm:mat16},
 then the main term in \eqref{eq:main'} equals 
 $$
 (C_{{\boldsymbol\rho},{\boldsymbol L}} + o(1)) \prod_{j=1}^r (\log B)^{m(\rho_i) - 1} ,
 $$
 for some absolute constant $C_{{\boldsymbol\rho},{\boldsymbol L}}$ that depends at most on
 $\rho_1 \dots, \rho_r$ and $L_1, \dots, L_r$.
 Further, $C_{\boldsymbol{\rho}, \mathbf{L}} > 0$ 
 if and only if there exists $\xx \in \ZZ^{n+1}$ 
 with $$\rho_1(L_1(\xx)) \cdots \rho_r(L_r(\xx)) > 0.$$
\end{proposition}
 
\begin{remark}[Leading constant] \label{rem:leading-constant}
 The proof yields the following information on the leading constant.
 Let $\mathcal{E}_{\rho_i}$ denote the set of primitive Dirichlet characters $\chi$ such that
 $m(\rho_i) = m(\chi \rho_i)$, 
 let $q_{\rho_i}$ denote the least common multiple of conductors of the elements of $\mathcal{E}_{\rho_i}$,
 and let $H \in \NN$ be such that Definition \ref{def:e-weak_frobenian}(3) holds for all the $\rho_i$.
 If $B_0 \geq 1$ is sufficiently large in terms of $r$, $s$, $H$, $q_{\rho_1}, \dots, q_{\rho_r}$ and the coefficients
 of linear forms $L_i(\xx)-L_i(\boldsymbol{0})$, then
 \begin{align*}
  &\sum_{\xx \in (B\mathfrak{K} + \boldsymbol{a}) \cap \ZZ^{n+1}} \rho_1(L_1(\xx)) \cdots \rho_r(L_r(\xx)) \\
  &= B^{n+1} \Big(C_{{\boldsymbol\rho},{\boldsymbol L}}^* \vol \fK^+ + O(B_0^{-1/2})+ o_{B \to \infty}(1)\Big)
    \prod_{j=1}^r \frac{1}{\log B}
    \prod_{B_0 < p \leq B} \Big(1 + \frac{\rho_j(p)}{p} \Big)
 \end{align*}
 for all $B>B_0$,
 where
 \begin{align*}
 & C_{{\boldsymbol\rho},{\boldsymbol L}}^* =
\bigg(\prod_{j=1}^r 
 \frac{e^{-\gamma m(\rho_j)}}{\Gamma(m(\rho_j))}
\bigg) 
\sum_{\substack{b_1, \dots, b_r \\ b_i \in (\ZZ/ q_{\rho_i}\ZZ)^*}}
 \bigg(\prod_{j=1}^r 
 \sum_{\chi \in \mathcal{E}_{\rho_j}(q_{\rho_j})} \bar\chi(b_j) 
 \bigg)
  \bigg(\prod_{p \leq B_0}(1-p^{-1})^{-r} 
  \bigg)
\\
 & \times 
 \sum_{\substack{u_1, \dots, u_r \\ p\mid u_i \Rightarrow p \leq B_0}}
 \sum_{\substack{A_1,\dots,A_r 
 \\ \in (\ZZ/Q_0 \ZZ)^*:
 \\ A_i \equiv b_i \Mod{q_{\rho_i} } }}
 \bigg(\prod_{i=1}^r 
 \rho_i(u_i)  
 \bigg) 
 \frac{1}{(uQ_0)^s}
 \sum_{\substack{\bv \in \\ (\ZZ/u Q_0 \ZZ)^s}}
 \prod_{j=1}^r 
 \1_{L_j(\bv) \equiv u_j A_j \Mod{u_j Q_0}}
\end{align*}
with $Q_0 = \prod_{p\leq B_0} p^{1+ v_p(q)}$, where
$q = \prod_{i=1}^r \prod_{\chi \in \mathcal{E}_{\rho_i}} q_{\chi}$,
and $u = \lcm(u_1, \dots, u_r)$.
Note that the character sums that appear in the expression for
$C_{{\boldsymbol\rho},{\boldsymbol L}}^*$ prevent us from being able to factorise 
this expression as a product over primes in general.
\end{remark}

\begin{proof}[Proof of Theorem \ref{thm:frob} assuming Proposition \ref{prop:right-order-asymptotics}]
 In view of the proposition, it suffices to prove that $C_{\fK, \boldsymbol{\rho}, \mathbf{L}} = 0$ if
 $\vol \fK^+=0$, and that $\vol \fK^+>0$ if and only if there exists 
 $\yy \in \fK$ such that $L_j(\yy) > L_j(\boldsymbol{0})$ for all $1 \leq j \leq r$.
 The latter part is clear from the definition of $\fK^+$ and continuity.
 The former part follows from the proof of Theorem \ref{thm:mat16}: 
 In the notation of the proof, we have
 $$\#~ \ZZ^{n+1} \cap ((N\fK + \boldsymbol{a})\Delta N\fK \cup \fK_N^+ \Delta N\fK^+) \ll N^n,$$
 and, since $\fK^+ \subset \RR^{n+1}$ is convex, it follows from a volume-packing argument 
 (see \cite[Appendix A]{GT-linearprimes}) 
 that
 $\#~ \ZZ^{n+1} \cap (N\fK^+) = \vol_{n+1}\fK^+ + O(N^n) = O(N^n)$ if 
 $\vol \fK^+ = \vol_{n+1} \fK^+= 0$.
 Thus, if $\vol \fK^+=0$, the same Cauchy-Schwarz application as in that proof shows that  
 the contribution of all of $N\fK + \boldsymbol{a}$ can be included in the error term, i.e.\ the main term is zero.
\end{proof}

Assuming that $\rho$ is a non-negative frobenian multiplicative function,
the following lemma provides an asymptotic formula for the mean values 
$S_{\rho}(N; \widetilde{W}(N),A)$ that appear in the main term \eqref{eq:main'}.
It will later be used to `lift' any given `starting point' $\xx$ for which 
$\rho_1(L_1(\xx)) \dots \rho_r(L_1(\xx)) > 0$ 
and deduce from the existence of such a point that the main term in \eqref{eq:main'} 
is of the correct order of magnitude.

\begin{lemma}[Lifting property] \label{lem:lifting}
Let $\rho$ be a non-negative frobenian multiplicative function with $m(\rho) > 0$, let
$\mathcal{E}_{\rho}$ denote the set of characters from Lemma \ref{lem:rho-in-APs},
and let $q_{\rho}$ denote the least common multiple of all conductors of characters from this set.
Let $N>1$ be a parameter and let $\widetilde{W}= \widetilde{W}(N)$ be as in Theorem \ref{thm:mat16},
in particular $q_{\rho}\mid\widetilde{W}$.
If further $\gcd(A,\widetilde{W})=1$, then
\begin{align*}
 &S_{\rho}(N; \widetilde W(N), A) = 
 \\
 &  \bigg( \frac{e^{-\gamma m(\rho)}}{\Gamma(m(\rho))}
    \sum_{\chi \in \mathcal{E}_{\rho}(q_{\rho})} \bar\chi(A) + o_{N\to \infty}(1)\bigg)
    \frac{\widetilde W}{\phi(\widetilde W)} 
    \frac{1}{\log N}
    \prod_{p \leq N, p \nmid \widetilde W}\Big(1+ \frac{\rho(p)}{p} \Big),
\end{align*}
where $\gamma$ denotes Euler's constant. 
The leading constant above is real and non-negative and it only depends on the residue class of $A \Mod{q_{\rho}}$.
If $A=1$, the constant is positive.
\end{lemma}

\begin{remark}
Since the leading constant above only takes values in a finite set, which is determined by $m(\rho)$ and 
$q_{\rho}$, we have 
$$
S_{\rho}(N; \widetilde{W}(N),A) \gg 
\frac{\widetilde{W}}{\phi(\widetilde{W})}\frac{1}{\log N} 
\prod_{p \leq N, p \nmid \widetilde W}\Big(1+ \frac{\rho(p)}{p} \Big)
$$
for all $N \geq 3$ and for all reduced residues $A \Mod{\widetilde W(N)}$
for which $\sum_{\chi \in \mathcal{E}_{\rho}} \bar\chi(A)$ is positive.
\end{remark}

\begin{proof}
By Lemma \ref{lem:rho-in-APs} we have, as $N \to \infty$,
\begin{align*}
 S_{\rho}(N; \widetilde W, A)
 = \frac{\widetilde W}{\phi(\widetilde W)} 
   \sum_{\chi \in \mathcal{E}_{\rho}(q_{\rho})} \bar\chi(A) 
   \frac{1}{N} \sum_{n \leq N} \rho(n) \chi^*(n) 
   + o(E_{\rho}(N;\widetilde W)),
\end{align*}
where $\widetilde{W} = \widetilde{W}(N)$ and where $\chi^* \Mod{\widetilde{W}}$ denotes the character
induced by $\chi$.
We seek to relate each of the mean values 
$S_{\rho \chi^*}(N) = \frac{1}{N} \sum_{n \leq N} \rho(n) \chi^*(n)$ in this expression 
to $S_{\rho} (N)$.
To start with, let $\rho^*$ denote the completely multiplicative function whose values at primes 
are given by $\rho^*(p) = \rho(p) \1_{p \not\in S}$.
Increasing $S$ if necessary, we may suppose that $S$ contains all primes $p \leq H^{8}$.
By Lemma \ref{lem:sum-rho-chi}, $S_{\rho \chi^*}(N)$ thus equals:
\begin{align*}
 \frac{1}{N}\sum_{\substack{n\leq N \\ (n, \widetilde{W})=1}}
 \rho(n)\chi(n)
 = S_{\rho^* \chi}(N) 
 \prod_{p \mid \widetilde{W}, p \not\in S} \left(1 - \frac{\chi(p)\rho(p)}{p}\right)
 +o\bigg(\frac{\phi(\widetilde W)E_{\rho\chi}(N,\widetilde{W})}{\widetilde W}\bigg).
\end{align*}
Concerning the mean value $S_{\rho^* \chi}(N)$, 
recall that $m(\rho) = m(\rho\chi)$ for all $\chi \in \mathcal{E}_{\rho}$, 
and note that Lemma \ref{lem:finitely-many}(3) thus implies
\begin{align*}
 \sum_{\substack{p \text{ prime},\, p\not\in S}} 
 \frac{\rho(p) - \re \chi(p)\rho(p)}{p} < \infty.
\end{align*}
It therefore follows from Elliott \cite[Theorem 4]{Elliott} that
\begin{align*}
 S_{\rho^* \chi}(N)
  = S_{\rho^*}(N) 
   \prod_{p\leq N,\, p\not\in S} 
   \frac{1 - \rho(p)p^{-1}} {1 - \chi(p)\rho(p)p^{-1}}
   + o\left(S_{\rho^*}(N) \right)
\end{align*}
as $N \to \infty$.
Since $m(\rho) = m(\rho\chi)$, it further follows from Lemma \ref{lem:primes} that
\begin{align*}
\prod_{w(N)\leq p\leq N} \left(1 - \frac{\chi(p)\rho(p)}{p}\right)
= (1+o(1))\left(\frac{\log N}{\log w(N)}\right)^{-m(\rho)},
\end{align*}
where the leading constant is $1$ and no longer depends on $\chi$,
and that 
$$ \prod_{p \mid \widetilde{W}, p \not\in S} \left(1 - \frac{\chi(p)\rho(p)}{p}\right)
\ll \prod_{p \leq w(N), p \not\in S} \left(1 - \frac{\rho(p)}{p}\right)
\ll (\log w(N))^{-m(\rho)}.
$$
Thus, combining all the above, we obtain
\begin{align*}
 S_{\rho \chi^*}(N)
 &= (1+o(1)) S_{\rho^*}(N) \left(\frac{\log N}{\log w(N)}\right)^{m(\rho)}
   \prod_{p\leq N, p \not \in S} \left(1 - \frac{\rho(p)}{p}\right) \\
 &\quad  + o\left(S_{\rho^*}(N) (\log w(N))^{-m(\rho^*)} \right)\\
 &=  \frac{e^{-\gamma m(\rho)}}{\Gamma(m(\rho))} \frac{1}{\log N}
  \left(\frac{\log N}{\log w(N)}\right)^{m(\rho)} 
  + o\left(\frac{(\log N)^{m(\rho)-1}}{ (\log w(N))^{m(\rho)}} \right),
\end{align*}
where we applied Wirsing's theorem \cite[Satz 1.1]{Wirsing} together with
the identity $m(\rho^*) = m(\rho)$ to $S_{\rho^*}(N)$ in order to express the leading constant explicitly.
Hence, the average value of $\rho$ in progressions modulo $\widetilde{W}(N)$ satisfies:
\begin{align*}
 &S_{\rho}(N; \widetilde W(N), A) \\
 &= \frac{\widetilde W}{\phi(\widetilde W)} 
   \frac{e^{-\gamma m(\rho)}}{\Gamma(m(\rho))} \frac{1}{\log N}
  \left(\frac{\log N}{\log w(N)}\right)^{m(\rho)}
  \sum_{\chi \in \mathcal{E}_{\rho}(q_{\rho})} \bar\chi(a) 
  + o(E_{\rho}(N;\widetilde W)) \\
 &= \frac{\widetilde W}{\phi(\widetilde W)} 
   \frac{e^{-\gamma m(\rho)}}{\Gamma(m(\rho))} \frac{1}{\log N}
  \bigg(\prod_{p \leq N, p \nmid \widetilde W}\Big(1+ \frac{\rho(p)}{p} \Big)\bigg)
  \sum_{\chi \in \mathcal{E}_{\rho}(q_{\rho})} \bar\chi(a) 
  + o(E_{\rho}(N;\widetilde W)),
\end{align*}
as claimed.
\end{proof}

\begin{proof}[Proof of Proposition \ref{prop:right-order-asymptotics}]
Our aim is to evaluate the main term of the asymptotic formula from Theorem \ref{thm:mat16} for $s=n+1$,
$N = B$ and ${\boldsymbol \varphi} = \boldsymbol{L}$.

By Lemma \ref{lem:lifting} there exists for every 
$j \in \{1, \dots, r\}$ and $b_j \in (\ZZ/q_{\rho_j}\ZZ)^*$
a constant $C_{b_j}(\rho_j)\geq 0$ such that
\begin{align} \label{eq:ll-consequence-1}
S_{\rho_j}(N;\widetilde{W},A_j) 
= (C_{b_j}(\rho_j) +o(1)) \frac{\widetilde{W}}{\phi(\widetilde{W})}\frac{1}{\log N} 
\prod_{p \leq N, p \nmid \widetilde W}\Big(1+ \frac{\rho_j(p)}{p} \Big), 
\end{align} 
uniformly for all reduced residues $A_j \Mod{\widetilde{W}(N)}$ that satisfy $A_j \equiv b_j \Mod{q_{\rho_j}}$.
Thus, Lemma \ref{lem:lifting} allows us to reduce the task of evaluating the main term of \eqref{eq:main'} 
to that of evaluating the expression
\begin{align} \label{eq:a_i-reduction}
\sum_{\substack{w_1, \dots, w_r 
 \\ p\mid w_i \Rightarrow p\mid \widetilde W
 \\ w_i \leq (\log N)^{B_2}}}
\sum_{\substack{A_1,\dots,A_r  \in (\ZZ/\widetilde W \ZZ)^* \\ A_i \equiv b_i \Mod{q_{\rho_i}}}}
 \beta_{\boldsymbol L}(w_1A_1, \dots, w_rA_r) 
 \prod_{j=1}^r  \rho_j(w_j) 
\end{align}
for any given tuple $(b_1, \dots, b_r) \in (\ZZ/ q_{\rho_1}\ZZ)^* \times \dots \times (\ZZ/ q_{\rho_r}\ZZ)^*$
and for 
\begin{equation} \label{eq:beta_L}
\beta_{\boldsymbol L}(w_1A_1, \dots, w_rA_r)
=\frac{1}{(w\widetilde W)^s}
\sum_{\substack{{\boldsymbol v} \in \\ (\ZZ/w \widetilde W \ZZ)^s}}
\prod_{j=1}^r 
\1_{L_j({\boldsymbol v}) \equiv w_j A_j ~(w_j \widetilde W)} 
\end{equation}
with $w = \lcm(w_1, \dots, w_r)$.

Before considering \eqref{eq:a_i-reduction} more carefully, let us indicate how the 
existence of a point $\xx \in \ZZ^{n+1}$ with the property $\rho_1(L_1(\xx)) \cdots \rho_r(L_r(\xx)) > 0$
implies positivity of the leading constant:
By Lemma \ref{lem:lifting} we have $C_{b_j}(\rho_j) > 0$ in the special case when $b_j=1$ 
in \eqref{eq:ll-consequence-1}.
Given $\xx \in \ZZ^{n+1}$ as above, we write $w^*_j := L_j(\xx)$ for each $j \in \{1, \dots, r\}$, 
and assume that $N$ is sufficiently large so that $p\mid w^*_j$ implies $p \mid \widetilde{W}(N)$.
Since $\xx$ solves the system
$$L_j(\xx) = w^*_j, \quad \rho_j(w^*_j) > 0, \qquad (1 \leq j \leq r),$$
it follows, writing $w=\lcm(w_1, \dots, w_r)$ as before, that
\begin{align} \label{eq:ll-consequence-2}
 \sum_{\substack{w_1, \dots, w_r 
 \\ p\mid w_j \Rightarrow p < B_0}}  
 \sum_{\substack{A_1,\dots,A_r \\ \in (\ZZ/Q \ZZ)^* \\ A_j \equiv 1 \Mod{q_{\rho_j}}}}
 \hspace{-.2cm}
 \bigg(\prod_{i=1}^r 
 \rho_i(w_i)  \bigg)~
 \frac{1}{(w Q)^s}
 \sum_{\substack{\bv \in \\ (\ZZ/w Q \ZZ)^s}}
 \prod_{j=1}^r 
 \1_{\substack{L_j(\bv) \equiv w_j A_j \\ \,\, \Mod{w_j Q}}}
> 0
\end{align}
for every $Q \in \NN$, 
provided $B_0$ is sufficiently large for $(w^*_1, \dots, w^*_r)$ to appear in the outer sum.

Our analysis of \eqref{eq:a_i-reduction} will rest upon the fact that the quantities 
$\beta_{\boldsymbol L}$ are closely related to the 
notion of local divisor densities studied in \cite[p.1831]{GT-linearprimes} and \cite[\S5.2]{BM17}.
We proceed by recalling this notion as well as some of the properties essential to this proof. 
For this purpose, let $\boldsymbol{\varphi} = (\varphi_1, \dots, \varphi_r)$ be a finite complexity
system of linear polynomials.
If $\mathbf{c}=(c_1, \dots, c_r) \in \NN_0^r$ and $m = \max \{ c_1, \dots, c_r \}$, then 
the associated divisor density is defined to be
$$
\alpha_{\boldsymbol{\varphi}}(p^{c_1},\dots,p^{c_r}) := 
\frac{1}{p^{ms}} \sum_{\mathbf u\in (\ZZ/p^m \ZZ)^s}
\prod_{i=1}^{r} \mathbf 1_{\varphi_i(\mathbf{u}) \equiv 0 \Mod{p^{c_i}}}.
$$
Divisor densities can, away from finitely many primes, be asymptotically evaluated and we have
(cf. \cite[eqn (5.6)]{BM17}):
\begin{equation}\label{eq:ev-alpha-1}
\alpha_{\boldsymbol{\varphi}}(p^{c_1},\dots,p^{c_r}) \begin{cases}
=1, &\mbox{if $n(\mathbf c)=0$,}\\
= p^{-\max_i \{c_i\}}, &  \mbox{if $p\gg_L 1$ and $n(\mathbf c)=1$,}\\
\leq p^{-\max_{i\neq j}\{c_i+c_j\}},
      & \mbox{if $p\gg_L 1$ and $n(\mathbf c)>1$,}\\
\ll_{L} p^{-\max_i \{c_i\}}, &
 \mbox{otherwise,}
\end{cases}
\end{equation}
where $n(\mathbf c)$ denotes the number of non-zero components of $\mathbf c$,
where
$$L = \max_{1 \leq i \leq r} \{\|\varphi_i\|,r,s\}$$
and where $\|\varphi_i\|$ denotes the maximum modulus of the coefficients of
$\varphi_i$.
Extending $\alpha_{\boldsymbol{\varphi}}$ multiplicatively, it follows from \eqref{eq:ev-alpha-1} that
\begin{align*} 
 \alpha_{\boldsymbol{\varphi}}(n_1,\dots,n_r) 
 \ll_L (\lcm(n_1, \dots, n_r))^{-1} 
 \ll_L (\max_j n_j)^{-1}.
\end{align*}
These divisor densities are linked to expressions of the form \eqref{eq:beta_L} for 
$\boldsymbol L=\boldsymbol{\varphi}$ through the following identity.
Let $m \geq a + \max(c_1, \dots, c_r)$ be an integer, then:
\begin{align} \label{eq:alpha-beta-relation}
&\sum_{\substack{a_1, \dots, a_r\\ \in (\ZZ/ p^a \ZZ)^*}}
\frac{1}{p^{ms}}
\sum_{\bv \in (\ZZ/p^m \ZZ)^s} 
\prod_{j=1}^r \1_{\varphi_j(\bv) \equiv p^{c_j} a_j \Mod{p^{a+c_j}}} \\
\nonumber
&=
\frac{1}{p^{ms}}
\sum_{\bv \in (\ZZ/p^m \ZZ)^s} 
\prod_{j=1}^r \1_{p^{c_j} \| \varphi_j(\bv) } 
= \sum_{\eps_1, \dots, \eps_r \in \{0,1\}}
(-1)^{\eps_1 + \dots +\eps_r}
\alpha_{\boldsymbol{\varphi}}(p^{c_1+\eps_1}, \dots, p^{c_r+\eps_r}).
\end{align}

Returning to the expression \eqref{eq:a_i-reduction}, let us consider for any fixed tuple
$(w_1, \dots, w_r)$ the sum over $(A_1, \dots, A_r)$.
By the Chinese remainder theorem, the function $\beta_{\boldsymbol L}$ is multiplicative.
Since the congruence conditions restricting the summation over $(A_1, \dots, A_r)$ 
only involve the finite set of prime factors dividing $q_{\rho_1} \dots q_{\rho_r}$, 
it follows by multiplicativity of $\beta_{\boldsymbol L}$ 
and from the first equality in \eqref{eq:alpha-beta-relation} that
\begin{align} \label{eq:A_i-decomp}
&\sum_{\substack{A_1,\dots,A_r  \in (\ZZ/\widetilde W \ZZ)^* \\ A_i \equiv b_i \Mod{q_{\rho_i}}}}
 \beta_{\boldsymbol L}(w_1A_1, \dots, w_rA_r)  \\
\nonumber 
&= 
\sum_{\substack{A_1,\dots,A_r  \in (\ZZ/\widetilde W \ZZ)^* \\ A_i \equiv b_i \Mod{q_{\rho_i}}}}
\frac{1}{(w\widetilde W)^s}
\sum_{\substack{{\boldsymbol v} \in \\ (\ZZ/w \widetilde W \ZZ)^s}}
\prod_{j=1}^r 
\1_{L_j({\boldsymbol v}) \equiv w_j A_j ~(w_j \widetilde W)} \\ 
\nonumber 
&= 
\Bigg(\sum_{\substack{A_1,\dots,A_r  \in (\ZZ/Q \ZZ)^* \\ A_i \equiv b_i \Mod{q_{\rho_i}}}}
\frac{1}{(uQ)^s}
\sum_{\substack{{\boldsymbol v} \in \\ (\ZZ/u Q \ZZ)^s}}
\prod_{j=1}^r \1_{L_j({\boldsymbol v}) \equiv w_j A_j ~(u_j Q)} \Bigg)\\
\nonumber
& \qquad \times
\prod_{p \mid \widetilde W, p \nmid Q}
\bigg(
\lim_{m \to \infty} 
\frac{1}{p^{ms}}
\sum_{\bv \in (\ZZ/p^{m}\ZZ)^s}
\left( \1_{p^{v_p(w_i)}\mid L_i(\bv)} - \1_{p^{v_p(w_i)+1}\mid L_i(\bv)} \right) \bigg)
\end{align}
whenever $Q$ is a divisor of $\widetilde{W}$ with the property that $\gcd(Q,\widetilde{W}/Q)=1$ and that 
$q_{\rho_i} | Q$ for every $i \in \{1, \dots, r \}$. 
Further, $u_i := \prod_{p\mid Q} p^{v_p(w_i)}$ and $u := \lcm(u_1, \dots, u_r)$ in the first factor above.

To handle the sum over $(w_1, \dots, w_r)$ in \eqref{eq:a_i-reduction}, we will use the decomposition
\begin{align} \label{eq:w_i-decomp}
\sum_{\substack{w_1, \dots, w_r 
 \\ p\mid w_i \Rightarrow p\mid \widetilde W(N)
 \\ w_i \leq (\log N)^{B_2}}}
= \sum_{\substack{w_1, \dots, w_r 
 \\ p\mid w_i \Rightarrow p\mid \widetilde W(N)}}
- \sum_{\substack{w_1, \dots, w_r 
 \\ p\mid w_i \Rightarrow p\mid \widetilde W(N)
 \\ \exists j. w_j > (\log N)^{B_2}}} 
\end{align}
together with the bound
\begin{align} \label{eq:auxiliary-bound}
 \sum_{\substack{w_1, \dots, w_r 
 \\ p\mid w_i \Rightarrow p\mid \widetilde W(N)
 \\ \exists j. w_j > (\log T)^{B_2}}}  
&\sum_{\substack{A_1,\dots,A_r 
 \\ \in (\ZZ/\widetilde W \ZZ)^*}}
 \bigg(\prod_{i=1}^r 
 \rho_i(w_i)  \bigg)
  \beta_{\boldsymbol L}(w_1A_1, \dots, w_rA_r)
\ll_{L} (\log N)^{-B_2/6 + o(1)}.
\end{align}

The latter bound is almost identical to the bound obtained in \cite[equations (11.3) and (11.4)]{Mat16}
and its proof is identical except for one step: 
the final line of [\emph{loc.\ cit.}, (11.4)] needs to be replaced by
$$
\ll_{L} (\log N)^{-B_2/6}~
2^{\omega(\widetilde W(N))} 
\ll_{L} (\log N)^{-B_2/6 }~2^{\pi(\log \log N)}
\ll_{L} (\log N)^{-B_2/6 +o(1)}.$$

The bound \eqref{eq:auxiliary-bound} certainly is
\begin{equation} \label{eq:auxiliary-bound-2}
 o_{N\to \infty}(1) 
\bigg(\frac{\phi(\widetilde{W})}{\widetilde{W}}\bigg)^r 
\prod_{i=1}^r\prod_{p\mid\widetilde{W}} \left(1+\frac{\rho_i(p)}{p}\right)^{-1},
\end{equation}
which will be enough to handle the contribution of the second sum in \eqref{eq:w_i-decomp}
towards \eqref{eq:a_i-reduction}.
Turning towards the first sum in \eqref{eq:w_i-decomp}, we let $B_0 > 0$  
be sufficiently large in terms of $L$ so that the second and third bound of \eqref{eq:ev-alpha-1} 
apply to every $p \geq B_0$ in the case where $\boldsymbol{\varphi}=\boldsymbol{L}$. 
In addition, suppose that $B_0 > 2 r H^r$ and that 
$B_0 > P^+(q_{\rho_1} \dots q_{\rho_r})$ is bounded below by the largest prime factor of 
$q_{\rho_1} \dots q_{\rho_r}$.
Let $Q_0 = \prod_{p\leq B_0} p^{v_p(\widetilde{W}(N))}$ be the factor of $\widetilde{W}(N)$ that is composed
of small primes. 
Then it follows from \eqref{eq:A_i-decomp} that
\begin{align} \label{eq:CRT-appl}
&\sum_{\substack{w_1, \dots, w_r 
 \\ p\mid w_i \Rightarrow p \mid \widetilde W(N)}}
 \bigg( \frac{\widetilde W}{\phi( \widetilde W)} \bigg)^r
 \sum_{\substack{A_1,\dots,A_r 
 \\ \in (\ZZ/\widetilde W \ZZ)^*
 \\ A_i \equiv b_i \Mod{q_{\rho_i} }
 }}
 \bigg(\prod_{i=1}^r 
 \rho_i(w_i) \bigg)
 \beta_{\boldsymbol L}(w_1A_1, \dots, w_rA_r)
 \\
\nonumber
&= 
\prod_{\substack{p\mid \widetilde W(N)\\ p>B_0}}
\sum_{\substack{ a_1, \dots, a_r \\ \in \NN_0 }}
\prod_{i=1}^r \frac{\rho_i(p^{a_i})}{1-p^{-1}}
\bigg(
\lim_{m \to \infty} 
\frac{1}{p^{ms}}
\sum_{\bv \in (\ZZ/p^{m}\ZZ)^s}
\left( \1_{p^{a_i}\mid L_i(\bv)} - \1_{p^{a_i+1}\mid L_i(\bv)} \right) \bigg) \times \\
& \nonumber
 \sum_{\substack{u_1, \dots, u_r 
 \\ p\mid u_i \Rightarrow p<B_0}}
 \left( \frac{Q_0}{\phi(Q_0)} \right)^r
 \hspace{-.7cm}
 \sum_{\substack{A'_1,\dots,A'_r 
 \\ \in (\ZZ/Q_0 \ZZ)^*
 \\ A'_i \equiv b_i \Mod{q_{\rho_i} } }}
  \hspace{-.6cm}
 \bigg(\prod_{i=1}^r 
 \rho_i(u_i)  \bigg) 
 \frac{1}{(uQ_0)^s}
 \sum_{\substack{\bv \in \\ (\ZZ/u Q_0 \ZZ)^s}}
 \prod_{j=1}^r 
 \1_{L_j(\bv) \equiv u_j A'_j \Mod{u_j Q_0}},
\end{align}
where $u = \lcm(u_1, \dots, u_r)$ and 
$Q_0 = \gcd(\widetilde W , \prod_{p\leq B_0} p^{\infty})$.
Let $\beta_{P(B_0)}(b_1, \dots, b_r)$ denote the final factor above, i.e.\ 
the factor that involves all primes $p \leq B_0$ and starts with the summation in $(u_1, \dots, u_r)$.
By \eqref{eq:ll-consequence-2}, we have 
$$
\beta_{P(B_0)}(1, \dots, 1) > 0.
$$
For every $p\mid \widetilde W(N)$ with $p \geq B_0$, the relation \eqref{eq:alpha-beta-relation} 
shows that the contribution to the above factorisation can be rewritten as
\begin{align} \label{eq:p>B-factor}
\nonumber
&\sum_{\substack{a_1, \dots, a_r \\ \in \NN_0}}
\prod_{i=1}^r \frac{\rho_i(p^{a_i})}{1 - p^{-1}}
~\lim_{m \to \infty} 
\frac{1}{p^{ms}}
\sum_{\bv \in (\ZZ/p^{m}\ZZ)^s}
\left( \1_{p^{a_i}\mid L_i(\bv)} - \1_{p^{a_i+1}\mid L_i(\bv)} \right) \\
&= \sum_{\substack{a_1, \dots, a_r \\ \in \NN_0}}
\prod_{i=1}^r \frac{\rho_i(p^{a_i})}{1 - p^{-1}}
\sum_{\boldsymbol{\eps} \in \{0,1\}^r}
(-1)^{n(\boldsymbol{\eps})} 
\alpha_{\boldsymbol{L}}(p^{a_1 + \eps_1},\dots,p^{a_r+\eps_r}).
\end{align}
This expression can now be asymptotically evaluated with the help of \eqref{eq:ev-alpha-1}.
Note that whenever the bound \eqref{eq:ev-alpha-1} on 
$\alpha_{\boldsymbol{L}}(p^{a_1 + \eps_1},\dots,p^{a_r+\eps_r})$
takes the form $p^{-k}$ for a given $k > 1$, 
then there at most $2^rk^r$ admissible choices of 
$(a_1, \dots, a_r)$ and $(\eps_1, \dots, \eps_r)$, and for each of these choices we have
$|\rho_i(p^{a_i+\eps_i})| < H^k$ for every $i \in \{1, \dots, r\}$.
Thus, the expression \eqref{eq:p>B-factor} for $p\geq B_0$ equals
\begin{align*}  
& \left(1 - \frac{1}{p} \right) ^{-r} 
\bigg( 1 - \frac{r}{p} + \sum_{i=1}^r \frac{\rho_i(p)}{p} + 
O \bigg( \sum_{k\geq2} \frac{k^r H^{rk}}{p^k} \bigg) \bigg)\\
&= \prod_{i=1}^r 
   \left(1 + \frac{\rho_i(p) }{p} \right) 
   + O_{H,r}(p^{-2})
= (1 + O_{H,r}(p^{-2}))
\prod_{i=1}^r 
   \left(1 + \frac{\rho_i(p) }{p} \right),
\end{align*}
and is certainly non-zero as soon as $B_0$ is sufficiently large in terms of $H$ and $r$.
By possibly increasing its value, we may suppose that $B_0$ is sufficiently large for
$$
\prod_{p \geq B_0} (1 + O_{H,r}(p^{-2})) = 1 + O(B_0^{-1/2}) > 0
$$
to hold.

Taking everything together, that is, applying first Lemma \ref{lem:lifting} 
to \eqref{eq:main'} and then combining \eqref{eq:a_i-reduction}
with \eqref{eq:w_i-decomp}, the bounds \eqref{eq:auxiliary-bound} and \eqref{eq:auxiliary-bound-2},
as well as the above analysis of \eqref{eq:CRT-appl}, 
the main term of \eqref{eq:main'} is seen to satisfy:
\begin{align} \label{eq:lem3.11-asymp}
 \nonumber
 & \sum_{\substack{w_1, \dots, w_r 
 \\ p\mid w_i \Rightarrow p\mid \widetilde W
 \\ w_i \leq (\log N)^{B_2}}}
 \sum_{\substack{A_1,\dots,A_r \\ \in (\ZZ/\widetilde W \ZZ)^*}}
 \bigg(\prod_{j=1}^r 
 \rho_j(w_j) S_{\rho_j}\Big(T;\widetilde W,A_j\Big)\bigg)~
 \beta_{\boldsymbol L}(w_1A_1, \dots, w_rA_r) \\
 \nonumber
 &\sim  \sum_{\substack{w_1, \dots, w_r 
 \\ p\mid w_i \Rightarrow p\mid \widetilde W
 \\ w_i \leq (\log N)^{B_2}}}
 \sum_{\substack{b_1, \dots, b_r \\ b_i \in (\ZZ/ q_{\rho_i}\ZZ)^*}} 
 \sum_{\substack{A_1,\dots,A_r \\ \in (\ZZ/\widetilde W \ZZ)^* : \\ A_i \equiv b_i (q_{\rho_i})}}
 \frac{\widetilde W}{\phi(\widetilde{W})}
 \Bigg(\prod_{j=1}^r C_{b_j}(\rho_j) \frac{ \rho_j(w_j)}{\log N}  \prod_{p \leq N, p \nmid \widetilde W} 
 \Big(1 + \frac{\rho_j(p)}{p} \Big)
 \Bigg)\times\\
 \nonumber
 &\qquad \qquad \qquad \qquad \qquad \qquad \qquad \qquad \qquad \qquad \qquad \qquad 
 \times \beta_{\boldsymbol L}(w_1A_1, \dots, w_rA_r)\\
 \nonumber
 &\sim 
 C_{\rho_1, \dots, \rho_r}
 \prod_{j=1}^r \frac{1}{\log N}   \prod_{B_0 < p \leq N} 
 \Big(1 + \frac{\rho_j(p)}{p}\Big)\Big( 1 + O_{H,r}(p^{-2})\Big)\\
  &\sim (1 + O(B_0^{-1/2}))
 C_{\rho_1, \dots, \rho_r}
 \prod_{j=1}^r \frac{1}{\log N}   \prod_{B_0 < p \leq N} 
 \Big(1 + \frac{\rho_j(p)}{p}\Big),
\end{align} 
where
\begin{align*}
C_{\rho_1, \dots, \rho_r} &=
\sum_{\substack{b_1, \dots, b_r \\ b_i \in (\ZZ/ q_{\rho_i}\ZZ)^*}}
\beta_{P(B_0)}(b_1, \dots, b_r) C_{b_1}(\rho_1) \dots C_{b_r}(\rho_r)\\
&\geq \beta_{P(B_0)}(1, \dots, 1) C_{1}(\rho_1) \dots C_{1}(\rho_r)
>0.
\end{align*}
We note as an aside that on inserting the explicit expressions for the constants $C_{b_i}(\rho_i)$ 
from Lemma~\ref{lem:lifting} as well as the definition of $\beta_{P(B_0)}(b_1, \dots, b_r)$,
the above yields the information about the leading constant summarised in Remark \ref{rem:leading-constant}.

Finally, Lemma \ref{lem:primes} shows that \eqref{eq:lem3.11-asymp} is, in fact, equal to
$$(1+o(1)) C_{\boldsymbol{\rho},\boldsymbol{L}} \prod_{j=1}^r (\log N)^{m(\rho_j)-1},$$
for some constant $C_{\boldsymbol{\rho},\boldsymbol{L}}>0$, as claimed.
\end{proof}

\section{Families of varieties over $\PP^1$} \label{sec:Serre}
In this section we prove Theorem \ref{thm:Serre}. The upper bound is proved in \cite{LS16}, so it suffices to prove the lower bound.

\subsection{Detectors} \label{sec:detectors}
We first construct frobenian multiplicative functions for detecting the everywhere locally soluble fibres.

\subsubsection{Set-up} 
Let $\pi:V \to \PP^1$ be as in Theorem \ref{thm:Serre}. We fix a choice of primitive integer vector  $\by = (y_0,y_1)$ such that the fibre over $y = (y_0:y_1) \in \PP^1(\QQ)$ is smooth and  everywhere locally soluble; this exists by assumption. We assume for simplicity of exposition that $y_0 \neq 0$.

Let $\Theta(\pi)$ denote the set of rational points of $\PP^1$ which lie below the non-pseudo-split fibers of $\pi$. We let $U = \PP^1 \setminus \{\theta: \theta \in \Theta(\pi)\}$, so that the fibre over every point of $U$ is pseudo-split. For each $\theta \in \Theta(\pi)$ we let $L_\theta(x_0,x_1) \in \ZZ[x_0,x_1]$ be a primitive binary linear form whose zero locus in $\PP^1$ is $\theta$ and such that $L_\theta(\by) > 0$.

Let $S$ be a large finite set of primes such that there exists a smooth proper scheme $\mathcal{V} \to \Spec \ZZ_S$
whose generic fibre is isomorphic to $V$, together with a morphism $\pi: \mathcal{V} \to \PP^1_{\ZZ_S}$ which extends the map 
$V \to \PP^1_{\QQ}$. We will allow ourselves to increase $S$ in this section. In particular, let $\mathcal{U}$ be the complement in $\PP^1_{\ZZ_S}$ of the closure of $\Theta(\pi)$ in $\PP^1_{\ZZ_S}$. Then we may assume that the fibre over every element of $\mathcal{U}$ is pseudo-split. We also enlarge $S$ to include all primes that divide the resultants of any two of the $L_\theta$ and such that $\gcd(L_\theta(\by), p) = 1$ for all $p \notin S$.

For each $\theta \in \Theta(\pi)$, let $\QQ \subset k_\theta$ be a finite Galois extension containing the field of definition of every geometric irreducible component of $\pi^{-1}(\theta)$. We  let $\Gamma_\theta = \Gal(k_\theta/\QQ)$. We assume  $S$ contains all primes which ramify in the $k_\theta$.
Let
$$\delta_\theta(\pi) = \frac{\#\left\{ \gamma  \in \Gamma_\theta : 
\begin{array}{l}
\gamma \mbox{ fixes an irreducible component } \\
\mbox{of $\pi^{-1}(\theta)_{k_\theta}$ of multiplicity $1$}
\end{array}
\right\}}{\# \Gamma_\theta}$$
as in Definition \ref{def:Delta}, and
\begin{equation} \label{def:P_theta}
	\mathcal{P}_\theta = S \cup  \left\{p \notin S: 
	\begin{array}{l}
		\Frob_p \in \Gamma_\theta \mbox{ fixes an irreducible component} \\
		\mbox{of $\pi^{-1}(\theta)_{k_\theta}$ of multiplicity $1$}
	\end{array}
	\right\}.
\end{equation}
This set is frobenian of density $\delta_\theta(\pi)$. Moreover $\delta_\theta(\pi) > 0$, due to our assumption that each fibre contains an irreducible component of multiplicity $1$.
For each $\theta \in \Theta(\pi)$ we define a completely multiplicative function $\varpi_\theta$ via
$$\varpi_\theta(n) =
\begin{cases}
	1, & \forall p \mid n \text{ we have } p \in \mathcal{P}_\theta, \\
	0, & \text{otherwise}.
\end{cases}$$

\begin{lemma} \label{lem:varpi}
	Let $\theta \in \Theta(\pi)$. Then $\varpi_\theta$ is a frobenian multiplicative function of mean $\delta_\theta(\pi)$.
\end{lemma}
\begin{proof}
	Follows immediately from the definitions.
\end{proof}

\subsubsection{Large primes} \label{sec:large_primes}

We use these $\varpi_\theta$ to detect
whether a fibre is locally soluble at sufficiently large primes.

\begin{lemma} \label{lem:detector}
	On enlarging $S$ if necessary, the following holds. Let $(x_0,x_1) \in \ZZ^2$ be such that $\gcd(x_0,x_1)=1$.
	If $\prod_{\theta \in \Theta(\pi)} \varpi_\theta(L_\theta(x_0,x_1)) = 1$
	then $\pi^{-1}(x_0:x_1)$ has a $\QQ_p$-point for all $p \notin S$.
\end{lemma}
\begin{proof}
We claim that $(x_0:x_1) \bmod p$ lies below
	a split fibre. To see this, first suppose that $p \nmid \prod_{\theta \in \Theta(\pi)} L_\theta(x_0,x_1)$. Then $(x_0:x_1) \not\equiv \theta \bmod p$ for 
	all $\theta \in \Theta(\pi)$. Thus $(x_0:x_1) \bmod p \in \mathcal{U}$, hence 
	the fibre over $(x_0:x_1) \bmod p$ is pseudo-split by construction. But a pseudo-split scheme
	over a finite field is split, as required.
	Next assume  that $p \mid L_\theta(x_0,x_1)$ for some $\theta \in \Theta(\pi)$,
	so that $(x_0:x_1) \equiv \theta \bmod p$.
	Then $p \in \mathcal{P}_\theta$ as $\varpi_\theta(p)=1$,
	hence $\Frob_p \in \Gamma_\theta$ fixes an irreducible component of multiplicity
	one of the fibre.
	So the fibre over $(x_0:x_1) \bmod p$ is  split,	as required.

	Thus, on enlarging $S$ if necessary, the Lang--Weil estimates \cite{LW54} imply that $\pi^{-1}(x_0:x_1) \bmod p$
	contains a smooth $\FF_p$-point ($S$ may be chosen uniformly for all $(x_0,x_1)$, due to the uniformity 
	of the Lang--Weil estimates). Hensel's lemma therefore implies that the fibre $\pi^{-1}(x_0:x_1)$ has a $\QQ_p$-point,
	as required.
\end{proof}

We now fix a choice of $S$ satisfying the above properties.

\subsubsection{Real points and small primes} \label{sec:small_primes}
Recall that the fibre over $(y_0:y_1)$ is smooth and everywhere locally soluble. The implicit function theorem  implies that the fibre over any sufficiently close real point $(x_0:x_1) \in \PP^1(\RR)$ to $(y_0:y_1)$ has a real point. So there exists $\delta > 0$ such that if 
$|x_1/x_0 - y_1/y_0| < \delta$, then the fibre over $(x_0:x_1)$ has a real point. We choose $\delta$ sufficiently small so that $L_\theta(x_0,x_1)>0$.

By the $p$-adic implicit function theorem, a similar conclusion applies for primes $p \in S$. Therefore, shrinking $\delta$ if necessary, for all $p \in S$ and all $(x_0,x_1) \in \ZZ^2$, if $|x_1/x_0 - y_1/y_0|_p < \delta$, then the fibre over $(x_0:x_1)$ has a $\QQ_p$-point.

\subsubsection{Conclusion}

Putting everything together and passing to the affine cone, we obtain the following.
\begin{lemma} \label{lem:conclusion}
	We have
	$$
	N_{\mathrm{loc}}(\pi,B) \geq \frac{1}{2}  \sum_{\substack{(x_0,x_1) \in \ZZ^2 \\ |x_0|,|x_1| \leq B  \\ \gcd(x_0,x_1) = 1   
	 \\ |x_1/x_0 - y_1/y_0|_v < \delta \, \forall v \in S \cup \{\infty\} }} 
	 \prod_{\theta \in \Theta(\pi)} \varpi_\theta(L_\theta(x_0,x_1)).
	$$
\end{lemma}
The sum in Lemma \ref{lem:conclusion} is non-zero, as the term $(x_0,x_1) =(y_0,y_1)$ contributes non-trivially. Indeed, it clearly occurs in the range of summation. Moreover, we have $L_\theta(\by) > 0$ and  $p \mid L_\theta(\by) \implies p \in S$; but $\varpi_\theta(p) = 1$ for all $p \in S$. Combining these facts shows that the summand is non-zero in this case.

We have thus reduced to a problem on sums of frobenian multiplicative
functions evaluated at binary linear forms. Theorem \ref{thm:frob} does not immediately apply
due to the coprimality condition and the imposed local conditions. 
As it will cause us no additional difficulties, we proceed by obtaining a general technical result on handling the kind of sums 
appearing in Lemma \ref{lem:conclusion}. We also give a higher-dimensional version  to assist with later applications.

\begin{theorem} \label{thm:frob_applications}
	Let $L_1(\bx),\ldots, L_r(\bx) \in \ZZ[x_0,\ldots,x_n]$  be pairwise linearly independent linear forms.
	Let $\rho_1,\ldots, \rho_r$ be real-valued non-negative frobenian multiplicative functions
	which are completely multiplicative and satisfy $m(\rho_i) \neq 0$.
	Let $S$ be a finite set of primes and $1 >\delta > 0$.
	Assume that there exists a primitive integer vector $\by \in \ZZ^{n+1}$ such that $\rho_j(L_j(\by)) > 0$ for 	all $ j \in \{1, \dots, r\}$.
	Then there exists $C_{\delta,S,\boldsymbol{\rho}, \boldsymbol{L}}>0$ such that as $B \to \infty$
	$$\sum_{\substack{\bx \in \ZZ^{n+1}  \\ \max_i |x_i| \leq B \\ \gcd(\bx) = 1 \\
\max\limits_{v \in S \cup \{\infty\}}|x_i/x_0 - y_i/y_0|_v < \delta}} 
	 \prod_{j=1}^r \rho_j(L_j(\bx)) \sim C_{\delta,S,\boldsymbol{\rho}, \boldsymbol{L}}B^{n+1}\prod_{j=1}^r(\log B)^{m(\rho_j)-1}.$$
\end{theorem}
The result applies to the $\varpi_\theta$, as they are completely multiplicative.

\subsection{Proof of Theorem \ref{thm:frob_applications}}
Let $\rho = \prod_{j =1}^r \rho_j$ and let $N(B)$ be the sum appearing in Theorem \ref{thm:frob_applications}.

\subsubsection{M\"{o}bius inversion}
We first apply M\"{o}buis inversion. To simplify some later parts of the proof, we only do this to primes not in $S$. This gives
$$N(B) = 
\sum_{\substack{k \leq B \\ \gcd(k,S) = 1}} \mu(k) 
\sum_{\substack{\bx \in \ZZ^{n+1} \\ 
	\max_i |x_i| \leq B \\
	k \mid \gcd(\bx)  \\ 
	\gcd(\bx, S) = 1  \\ 
	\max\limits_{v \in S \cup \{\infty\}}|x_i/x_0 - y_i/y_0|_v < \delta}} 
	 \prod_{j=1}^r \rho_j(L_j(\bx)).
	 $$
Here we use the notation $\gcd(\bx, S) := \prod_{p \in S} \gcd(x_0,\dots,x_n,p)$. Using that $\rho$ is completely multiplicative and that the $L_i$ are homogeneous,
we  obtain
\begin{equation} \label{eqn:Mobius}
N(B)  = \sum_{\substack{k \leq B \\ \gcd(k,S) = 1}} \mu(k) \rho(k)
\sum_{\substack{\bx \in \ZZ^{n+1} \\ 
	\max_i |x_i| \leq B/k \\
	\gcd(\bx,S) = 1 \\ 
	\max\limits_{v \in S \cup \{\infty\}}|x_i/x_0 - y_i/y_0|_v < \delta}} 
	 \prod_{j=1}^r \rho_j(L_j(\bx)).
\end{equation}
As $\rho_j(n) \ll_\varepsilon n^{\varepsilon/2r}$, using $k \leq B$ we find that the inner sum above is
$$
\ll \sum_{\substack{(x_0,\ldots,x_n) \in \ZZ^{n+1}   \\ \max_i |x_i| \leq B/k }}  \prod_{j=1}^r \rho_j(L_j(\bx)) 
	 \ll_\varepsilon B^{\varepsilon/2} \sum_{\substack{(x_0,\ldots,x_n) \in \ZZ^{n+1} \\ \max_i |x_i| \leq B/k }}  1
	 \ll_\varepsilon B^{\varepsilon/2} \left(\frac{B}{k}\right)^{n+1}.
$$	 
This in particular shows that the contribution to \eqref{eqn:Mobius} from those $B^{\varepsilon} \leq k \leq B$ is $O_{\varepsilon}(B^{n+1+\varepsilon/2 - n\varepsilon}) = O(B^{n+1 - \varepsilon/2})$. This gives
\begin{equation} \label{eqn:MT}
	N(B)  = \sum_{\substack{k \leq B^\varepsilon \\ \gcd(k,S)= 1}} \mu(k) \rho(k) \hspace{-10pt}
\sum_{\substack{\bx \in \ZZ^{n+1} \\ 
	\max_i |x_i| \leq B/k \\
	\gcd(\bx,S) = 1 \\ 
	\max\limits_{v \in S \cup \{\infty\}}|x_i/x_0 - y_i/y_0|_v < \delta}} 
	\hspace{-10pt}
	 \prod_{j=1}^r \rho_j(L_j(\bx))
	 + O_\varepsilon(B^{n+1-\varepsilon/2}).
\end{equation}

\subsubsection{Removing the $p$-adic conditions}
We next deal with our $p$-adic conditions by rewriting them in terms of congruences.

\begin{lemma}
	There exists an integer $M$ and a subset $A \subseteq (\ZZ/M\ZZ)^{n+1}$
	such that
	$$\gcd(\bx,S) = 1 \mbox{ and } \, \forall i, \forall p \in S, |x_i/x_0 - y_i/y_0|_p < \delta \quad \iff \quad \bx \bmod M \in A.$$
\end{lemma}
\begin{proof}
Let $p \in S$ and choose the largest $m \in \ZZ$ such that $\delta \leq p^{-m+1}$.
 The condition $\gcd(\bx,p) = 1$ is equivalent to $p \nmid x_i$, for some $i$.  We first assume that $p \nmid x_0$. Here $x_i/x_0 \in \ZZ_p$, thus also $y_i/y_0 \in \ZZ_p$ when $|x_i/x_0 - y_i/y_0|_p < 1$. Therefore in this case our condition is equivalent to the congruence
 $$x_i \equiv (y_i/y_0)x_0 \bmod p^m,$$
 as claimed.
Now consider the case $p \mid x_0$, so that without loss of generality $p \nmid x_1$. Then $x_0/x_1 \in \ZZ_p$. If $|x_0/x_1|_p \neq |y_0/y_1|_p$ then we have
$$|x_1/x_0 - y_1/y_0|_p = \max\{|x_1/x_0|_p, |y_1/y_0|_p\} \geq |x_1/x_0|_p \geq 1 > \delta,$$
thus we must have $|x_0/x_1|_p = |y_0/y_1|_p$ under our condition. 
 As $x_1 \in \ZZ_p^*$ we obtain
$$|x_1/x_0 - y_1/y_0|_p = |x_1y_0 - y_1x_0|_p/|x_0y_0|_p =  |x_1y_0 - y_1x_0|_p\cdot |y_1/y_0^2|_p.$$
Letting $s = v_p(y_0^2/y_1)$, we find that  $|x_1/x_0 - y_1/y_0|_p < \delta$ is equivalent to 
$$x_0y_1 \equiv x_1 y_0\bmod p^{m+s}.$$
Now let $j > 1$. Then we have
$$|x_j/x_0 - y_j/y_0|_p = |(x_jy_1/x_1y_0)(1 + O(p^{m+s})) - y_j/y_0|_p.$$
Letting $r = v_p(y_0)$, using $p \nmid x_1$ and $r \leq s$, we find that $|x_j/x_0 - y_j/y_0|_p < \delta$ is equivalent to the congruence
$$x_jy_1 \equiv y_jx_1 \bmod p^{m+r}.$$
This handles all cases and proves the result for all $p \in S$. One then deduces the result from the Chinese remainder theorem.
\end{proof}

Using this lemma in \eqref{eqn:MT} we therefore obtain the main term
\begin{align*}
& \sum_{\substack{k \leq B^\varepsilon \\ \gcd(k,S)= 1}} \mu(k) \rho(k)
\sum_{\substack{\bx \in \ZZ^{n+1} \\ 
	\max_i |x_i| \leq B/k \\
	\bx \bmod M \in A \\ 
	\max\limits_{i}|x_i/x_0 - y_i/y_0| < \delta}} 
	 \prod_{j=1}^r \rho_j(L_j(\bx)) \\
	& = \sum_{\ba \in A }\sum_{\substack{k \leq B^\varepsilon \\ \gcd(k,S)= 1}} \mu(k) \rho(k) \hspace{-30pt}
\sum_{\substack{\bx \in \ZZ^{n+1} \\ 
	\max_i |x_i + a_i/M| \leq B/Mk \\
	\max\limits_{i}|(x_i + a_i/M)/(x_0 +a_0/M) - y_i/y_0| < \delta}} 
	\hspace{-20pt}
	 \prod_{j=1}^r \rho_j(ML_j(\xx) +  L_j(\boldsymbol{a})) 	 
\end{align*}
after summing over the elements of $A$ and making the obvious change of variables. Here we make the abuse of notation of identifying each element of $\ZZ/M\ZZ$ with its representative in $[0,M-1] \cap \ZZ$.

\subsubsection{Applying Theorem \ref{thm:frob}}
We now apply Theorem \ref{thm:frob} to the above sum with
$$\mathfrak{K}= 
\{ \xx \in \RR^{n+1} : |x_i| \leq 1, |x_i/x_0 - y_i/y_0|<\delta \},$$
with $\boldsymbol{a}$ replaced by $(-a_0/M,\dots,-a_n/M)$, and as $B/Mk \to \infty$.
The non-constant parts $ML_j(\xx)$ of the linear polynomials  are still pairwise linearly independent.
Moreover, as $L_j(\by) > 0$ we have $ML_j(\by) + L_j(\ba) > ML_j(\boldsymbol{0}) + L_j(\boldsymbol{a})$.
Thus, all assumptions of Theorem \ref{thm:frob} are satisfied. Since 
$\mathfrak{K}$ and the linear forms are independent of $k$, and $B/Mk \to \infty$ as $B \to \infty$ for $k \leq B^{\varepsilon}$,  we obtain
$$N(B) \sim  \sum_{\ba \in A} \sum_{\substack{k \leq B^\varepsilon \\ \gcd(k,M)=1}} \mu(k)\rho(k) C_{\mathfrak{K},\mathbf{a},M} \frac{B^{n+1}}{(Mk)^{n+1}} \prod_{j=1}^r\left(\log \frac{B}{Mk}\right)^{m(\rho_j)-1},$$
for some constant $C_{\mathfrak{K},\mathbf{a},M} \geq 0$ which is non-zero for $\ba \equiv \by \bmod M$, as follows from our assumption that
$\prod_{j=1}^r \rho_j(L_j(\by)) > 0$. Expanding out gives
$$\prod_{j=1}^r\left(\log \frac{B}{Mk}\right)^{m(\rho_j)-1} = \left(1 + O(\log k / \log B)\right)  \prod_{j=1}^r\left(\log B \right)^{m(\rho_j)-1}.$$
The resulting error term here is a satisfactory since the sum $\sum_{k} |\mu(k)|\rho(k)\log k /k^{n+1}$ is convergent. As for the main term, the leading constant is given by
$$
\sum_{\ba \in A} \frac{C_{\mathfrak{K},\mathbf{a},M}}{M^{n+1}}  \sum_{\substack{k =1 \\ \gcd(k,M)=1}}^\infty \frac{\mu(k)\rho(k)}{k^{n+1}}.
$$
To show positivity of the leading constant in Theorem \ref{thm:frob_applications} it suffices to note that
$$\sum_{\substack{k =1 \\ \gcd(k,M) = 1}}^\infty \frac{\mu(k)\rho(k)}{k^{n+1}} = \prod_{\gcd(p,M)=1}\left(1 - \frac{\rho(p)}{p^{n+1}}\right)$$
is positive. Indeed, this Euler product is absolutely convergent so it suffices to show each Euler factor is non-zero. But it is easily checked that $\rho$ being completely multiplicative and Definition \ref{def:class_F} 
implies that $|\rho(p)| \leq 1$ for all primes $p$, as otherwise
this would contradict $\rho(n) \ll_{\varepsilon} n^{\varepsilon}$.
This proves Theorem \ref{thm:frob_applications}.
\qed

\medskip \noindent
Theorem \ref{thm:Serre} now follows from Lemma \ref{lem:varpi}, Lemma \ref{lem:conclusion} and Theorem \ref{thm:frob_applications}.

\begin{remark}
	The proof of Theorem \ref{thm:Serre} shows the following stronger statement. For any $y \in \PP^1(\QQ)$ with $\pi^{-1}(y)$  smooth and everywhere locally soluble, any finite set of places $S$ and any open neighbourhoods $y \in U_p \subset \PP^1(\QQ_p)$, we have
	$$
	\#\{x \in \PP^1(\QQ): H(x) \leq B, x \in \pi(V(\Adele_\QQ)), x \in U_p \, \forall p \in S\} \asymp \frac{B^{n+1}}{(\log B)^{\Delta(\pi)}}.
	$$
	This stronger statement is useful for applications,
	and can be viewed as a version of weak approximation.
	We will require this for the proof of Theorem \ref{thm:RP}.
\end{remark}

\begin{remark} \label{rem:const_BM}
	Let us now give an example of a global obstruction  to the positivity of the leading constant in Theorem \ref{thm:frob} (cf.~Remark \ref{rem:const}).
	
	Let $V$ be a smooth projective variety over $\QQ$ with a morphism 
	$\pi:V \to \PP^1$ whose generic fibre is rationally connected and 
	such that each non-split fibre lies over a rational point. Assume that
	$V$ fails the Hasse principle, but each smooth fibre of $\pi$
	satisfies the Hasse principle. 
	(See \cite[Prop.~7.1]{CTCS80} for an explicit example coming
	from a Brauer--Manin obstruction.)
	
	Let $S$ be a finite set of places of $\QQ$.
	The argument  in \S\ref{sec:small_primes} applies in this case,
	under our weaker assumption that only $V(\Adele_\QQ) \neq \emptyset$,
	and shows that there is a smooth fibre which is soluble at all places 
	in $S$. Taking $S$ sufficiently large and choosing such a point
	$y \in \PP^1(\QQ)$, analogously to Lemma \ref{lem:conclusion} we have 	
$$
	N_{\mathrm{loc}}(\pi,B) \geq \frac{1}{2}  \sum_{\substack{(x_0,x_1) \in \ZZ^2 \\ |x_0|,|x_1| \leq B  \\ \gcd(x_0,x_1) = 1   
	 \\ |x_1/x_0 - y_1/y_0|_v < \delta \, \forall v \in S }} 
	 \prod_{\theta \in \Theta(\pi)} \varpi_\theta(L_\theta(x_0,x_1)).
	$$
There is no local obstruction here to the vanishing of the leading constant, in the following sense: Recall that the frobenian multiplicative functions $\varpi_\theta$ satisfy $\varpi_\theta(p) = 1$ for all $p\in S$, so there is clearly no obstruction for such $p$. For $p \notin S$, providing $S$ is sufficiently large, there exists $\bx \in \ZZ^2$ such that $p \nmid \prod_{\theta \in \Theta(\pi)} L_\theta(x_0,x_1)$. Thus the $p$-adic component of $\prod_{\theta \in \Theta(\pi)} L_\theta(x_0,x_1)$ is just a unit, hence the $p$-adic part of  $\varpi_\theta$ equals $1$ in this case as well.
	
	But the leading constant in Theorem \ref{thm:frob_applications} must be zero here; indeed $\pi$ has no smooth everywhere locally soluble fibre, since otherwise this fibre would have a rational point, which contradicts that $V$ has no rational point.	
	Thus here there is no local obstruction to the vanishing of the leading constant in Theorem \ref{thm:frob}, but there is a global obstruction coming from a failure of the Hasse principle. 
		These observations show that, in general, there is no simple local-global principle, nor a simple condition involving a finite set of places $S$, for the positivity of the leading constant in Theorem \ref{thm:frob}.
	The crucial assumption in Theorem \ref{thm:Serre} that there is an everywhere locally soluble smooth fibre is required to show the positivity of the leading constant in our application of Theorem \ref{thm:frob_applications}. 
	
	The above construction uses the fact that $\varpi_\theta(p) = 1$ for all $p\in S$.
	Comparing with the expression from Remark \ref{rem:leading-constant} for the leading constant in Theorem \ref{thm:frob_applications},
	one might be tempted to think that the factor $C^*_{\boldsymbol{\rho}, \boldsymbol{L}}$ 
	can be forced to be positive by a condition of the form $\rho_j(p) = 1$ for all $p \in S$ and 
	$1 \leq j \leq r$, if $S$ is sufficiently large to include all primes $p \leq B_0$. 
	In this case, the asymptotic formula stated in Remark \ref{rem:leading-constant} would imply a 
	local-global principle.
	
	We point out that this line of reasoning does not apply to the situation above. 
	In fact, the linear polynomials $ML_j(\xx) +  L_j(\boldsymbol{a})$ that we apply Theorem \ref{thm:frob_applications} to depend on $M$, and therefore on the set $S$.
	Thus, the parameter $B_0$ in Remark \ref{rem:leading-constant} needs to be sufficiently large in terms of not only $H$, $r$ and $\boldsymbol{L}$, 
	but also in terms of $S$ in order to be able to decide positivity of the leading constant.
	Hence, the factor $C^*_{\boldsymbol{\varpi}, M\boldsymbol{L}+\boldsymbol{a}}$ involves primes outside 
	of $S$ at which the functions $\varpi_{\theta}$ are not trivially equal to $1$.
\end{remark}

\section{Controlling failures of the Hasse principle} \label{sec:RP}
In this section we prove Theorem \ref{thm:RP}. We first recall some facts about Brauer groups and the Brauer--Manin obstruction.

\subsection{The Brauer group} \label{sec:Brauer_basic}
Let $V$ be a regular integral Noetherian scheme. 
\subsubsection{Residues}
We define the (cohomological) Brauer group of $V$ to be
$\Br V = \HH^2(V,\Gm)$. A theorem of Grothendieck \cite[Prop.~6.6.7]{Poo17} states that the natural map 
$
	\Br V \to \Br \kappa(V)
$
is injective, where $\kappa(V)$ denotes the function field of $V$.
This in particular shows that $\Br V$ is a torsion group, so that 
\begin{equation} \label{eqn:ell}
	\Br V = \bigoplus_{\text{primes } \ell}\Br V\{\ell\}.
\end{equation}
Let $D \in V^{(1)}$. If $\ell$ is a prime which is invertible on $V$,
then there is a residue map $\res_D: \Br \kappa(V)\{\ell\} \to \HH^1( \kappa(D), \QQ_\ell/\ZZ_\ell)$.
Using \eqref{eqn:ell}, one defines the residue $\res_D(b) \in \HH^1( \kappa(D), \QQ/\ZZ)$ of any element
$b \in \Br \kappa(V)$ whose order is invertible on $V$. We say that $b$ is 
\emph{unramified} at $D \in V^{(1)}$ if $\res_D (b) = 0$.
The residue maps give rise (see \cite[\S 6.8]{Poo17} for details) to an exact sequence
\begin{equation} \label{seq:purity}
	0 \to \Br V\{\ell\} \to \Br \kappa(V)\{\ell\} \to \bigoplus_{D \in V^{(1)}} \HH^1( \kappa(D), \QQ_\ell/\ZZ_\ell).
\end{equation}

\subsubsection{Brauer--Severi schemes} \label{sec:BS}
To any Brauer--Severi scheme $\pi: B \to V$ one may associate a Brauer group element $\alpha \in \Br V$. This construction is such that $\pi$ has a section if and only if the class of $\alpha$ is trivial in $\Br V$. In particular, for $P \in V$, we have $\alpha(P) = 0$ if and only if $\pi^{-1}(P)$ has a $\kappa(P)$-rational point.

\subsubsection{Filtration}
If $V$ is defined over a field $k$, then we define the algebraic part  of the Brauer group of $V$ to be
$ \Br_1 V = \ker( \Br V \to \Br V_{\bar{k}})$
. The map $\Br k \to \Br V$ need not be injective in general, however it is injective
if $V(k) \neq \emptyset$.
An element of $\Br V$ which does not lie in $\Br_1 V$ is called \emph{transcendental}.

\subsubsection{The Brauer--Manin obstruction}
We recall some facts about the Brauer--Manin obstruction (see e.g.~\cite[\S8.2]{Poo17}).
We have the fundamental exact sequence
\begin{equation} \label{eqn:CFT}
	0 \to \Br \QQ \to \bigoplus_{v} \Br \QQ_v \to \QQ/\ZZ \to 0,
\end{equation}
where the direct sum is over the places $v$ of $\QQ$. The last map is given by the sum over all local invariants $\inv_v : \Br \QQ_v \to \QQ/\ZZ$.
Given a smooth variety $V$ over $\QQ$, there is a well-defined pairing
$$\Br V \times V(\Adele_\QQ) \to \QQ/\ZZ,\quad
(\alpha, (P_v)) \mapsto \sum_v \inv_v \alpha(P_v)$$
which is right continuous and trivial on the image of $V(\QQ)$. We denote the right kernel of a subset $\mathcal{A} \subset \Br V$ by $V(\Adele_\QQ)^\mathcal{A}$; note that $V(\QQ) \subset V(\Adele_\QQ)^\mathcal{A}$ by the fundamental exact sequence. For $\mathcal{A} = \Br V$ we simply write $V(\Adele_\QQ)^{\Br}$.

We say that the \emph{Brauer--Manin obstruction is the only obstruction to the Hasse principle for $V$} if the implication $V(\Adele_\QQ)^{\Br} \neq \emptyset \implies V(\QQ) \neq \emptyset$ holds.

\subsection{The result}

We prove the following generalisation of Theorem \ref{thm:RP}. 
\begin{theorem} \label{thm:RP2}
	Let $V$ be a smooth projective variety over $\QQ$ equipped with a morphism $\pi:V \to \PP^1$
	whose generic fibre is geometrically integral.
	Assume that each fibre of $\pi$ contains an irreducible component of multiplicity $1$ and that each
	non-split fibre of $\pi$ lies over a rational point.
	Assume also that
	$\HH^1(V_{\bar{\eta}}, \QQ/\ZZ) = \HH^2( V_{\bar{\eta}}, \OO_{V_{\bar{\eta}}}) = 0$,
	and that the Brauer--Manin obstruction
	is the only one to the Hasse principle for the smooth fibres of $\pi$. 
	If $V(\QQ) \neq \emptyset$ then
	$$N(\pi,B) \gg \frac{B^2}{(\log B)^{\omega(\pi)}},
	\quad \text{ for some } \omega(\pi)>0.$$
\end{theorem}

\begin{proof}
Our approach combines the method of proof of Theorem \ref{thm:Serre} with the techniques from the proof of \cite[Thm.~9.17]{HW16}, as well as some conceptual improvements on \emph{loc.~cit.} due to Colliot-Th\'el\`ene (cf.~the proof of \cite[Thm.~7.13]{CT15}).

Let $\mathcal{A} \subset \Br V_\eta$ be a set of representative of the elements of $\Br V_\eta/ \Br \kappa(\eta)$. 
This is finite by our assumptions and \cite[Lem.~8.6]{HW16}. Choose some dense open set $U \subset \PP^1$ such that $V_U:= V \times_U \PP^1$ is smooth and each element of $\mathcal{A}$ is defined on $V_U$. Choose  Brauer--Severi schemes $\psi_\alpha : Y_\alpha \to V_U$ representing each $\alpha \in \mathcal{A}$. We let $\psi: Y \to V$ be a  smooth projective compactification of the fibre product $\prod_\alpha Y_\alpha$ over $V_U$. 

For motivation, let us briefly explain how we \emph{would like} the proof to go. Let $P \in V_U(\QQ)$. As $\alpha(P) \in \Br \QQ$, we may change our choice of representative for $\alpha$ to assume that $\alpha(P) = 0$ for all $\alpha \in \mathcal{A}$. As explained in \S \ref{sec:BS}, this implies that $P \in \psi_\alpha(Y_\alpha(\QQ))$ for all $\alpha$, hence $P \in \psi(Y(\QQ))$ and so $Y_U(\QQ) \neq \emptyset$. We now apply Theorem \ref{thm:Serre} to  $\pi \circ \psi: Y \to \PP^1$, which gives the stated order of magnitude $x \in \PP^1(\QQ)$ such that $Y_x(\Adele_\QQ) \neq \emptyset$, with $\omega(\pi) = \Delta(\pi \circ \psi)$. For such $x$ we have $V_x(\Adele)^{\mathcal{A}} \neq \emptyset$ by the construction of $Y$; however for almost all $x$ the group $\mathcal{A}$ generates $\Br V_x/\Br \QQ$, thus $V_x(\Adele_\QQ)^{\Br} \neq \emptyset$ and so $V_x(\QQ) \neq \emptyset$ by our assumptions, as required.

The problem with this argument is that $\pi \circ \psi$ may \emph{not} satisfy the assumptions of Theorem \ref{thm:Serre}: despite the non-split fibres of $\pi$ lying over rational points, there may be new non-split fibres of $\pi \circ \psi$ which don't lie over rational points. We thus need to re-run the proof of Theorem \ref{thm:Serre}, paying careful attention to the new non-split fibres. This subtlety also arises in the proof of \cite[Thm.~9.17]{HW16}, and the method to deal with it  originated in work of Harari \cite[Lem.~4.1.1]{Har94}. This is quite a delicate argument that requires us to introduce more notation and work with a larger set of Brauer elements than $\mathcal{A}$. We have modified this approach to our setting, which manages to avoid the use of Harari's ``formal  lemma''.

Let $S'$ be a sufficiently large set of places.
Let $\theta_1,\dots, \theta_n \in \PP^1(\QQ)$ be the points below the non-split fibres of $\pi$ and $L_i$ the corresponding primitive binary linear forms. We let $\varpi_{\theta_i}$ be the frobenian multiplicative function obtained by applying the construction from \S\ref{sec:detectors} to the fibre of $\pi \circ \psi$ above $\theta_i$.  Let $\theta_{n+1},\dots,\theta_N$ denote those closed points of $\PP^1$ below the new non-split fibres of $\pi \circ \psi$. Let $k_i$ be the residue field of $\theta_i$ and $K_i/k_i$ the splitting field of the irreducible components of $Y_{\theta_i}$. Choose sufficiently large distinct primes $p_i \notin S'$ which are completely split in  $K_i$. Let $\Gamma_i \subset \Br U$ be a finite subgroup such that the image of the residue map at $\theta_i$

$$\res_{\theta_i}: \Gamma_i \to \HH^1(k_i, \QQ/\ZZ)$$
contains $\HH^1(K_i/k_i, \QQ/\ZZ)$; this exists by   assumption (9.9) from \cite[Thm.~9.17]{HW16}, which holds in our case as at least one of the $\theta_i$ is rational (see \cite[Rem.~9.18(ii)]{HW16}). We let $\Gamma = \sum_{i=n+1}^N \Gamma_i$ and set $\mathcal{A}' = \mathcal{A} \cup \pi^* \Gamma$. Then, as above, we choose  Brauer--Severi schemes $\psi_\alpha' : Y_\alpha' \to V_U$ representing each $\alpha \in \mathcal{A}'$ and let $\psi': Y' \to V$ be a  smooth projective compactification of the fibre product of the $\psi_\alpha'$.

We have assumed the existence of a rational point in $V(\QQ)$. To get the proof to work, we need to choose this point carefully. By  \cite[Thm.~9.28]{HW16} and our assumptions, the variety $V$ satisfies weak weak approximation. Namely, choosing $S'$ sufficiently large, the set $V(\QQ)$ is dense in $\prod_{p \notin S'}V(\QQ_p)$. Moreover, as $p_i$ is completely split in $K_i$ and sufficiently large, we have $V_{\theta_i}(\QQ_{p_i}) \neq \emptyset$. Thus there exists a rational point $P \in V_U(\QQ)$ such that $y = \pi(P)$ is arbitrarily close to $\theta_i$ with respect to $p_i$, for each $i$. This is our choice of rational point, which we fix.

As $P$ lies in $V_U$, the evaluation $\alpha(P) \in \Br \QQ$ of each $\alpha \in \mathcal{A}'$ is well-defined. We  may change our choices of representatives $\alpha$ by an element of $\Br \QQ$ if we wish. So without loss of generality, we may assume that $\alpha(P) = 0$ for all $\alpha \in \mathcal{A}'$. Then  $\alpha(P) = 0$ implies that $P \in \psi'_\alpha(Y'_\alpha(\QQ))$. It follows that $P \in \psi'(Y'(\QQ))$. 

For $x \in \PP^1(\QQ)$, we let $\Omega_{x,i} = \{p \notin S' : x \bmod p \in \theta_i \bmod p\}$.
We now apply Theorem \ref{thm:frob_applications} to $\varpi_1(L_1(\bx)),\dots, \varpi_n(L_n(\bx))$ with $S = S' \cup \{p_{n+1},\dots, p_N\}$ and the chosen $y$. As in the proof of Theorem \ref{thm:Serre}, for all sufficiently small $\delta$ we obtain
\begin{equation}
	\#\{x \in U(\QQ): H(x) \leq B, x \text{ satisfies \eqref{eqn:Harari}}\}
	\gg B^2\prod_{i=1}^n(\log B)^{m(\varpi_i)-1} \label{eqn:HW}
\end{equation}
where
\begin{equation} \label{eqn:Harari}
	|x_1/x_0 - y_1/y_0|_v < \delta \,\, \forall v \in S \cup \{\infty\}, \quad Y'_x(\QQ_p) \neq \emptyset \,\, \forall p \notin \Omega_{x,n+1},\dots, \Omega_{x,N}.
\end{equation}
Let us clarify that we are only applying the method for $i =1,\dots,n$, so we do not claim local solubility at the primes in $\Omega_{x,n+1},\dots, \Omega_{x,N}$.
The leading constant in Theorem \ref{thm:frob_applications} is non-zero in this case, due to the existence of $y$.

Fix now $x$ satisfying \eqref{eqn:Harari}. For $p \notin \Omega_{x,n+1},\dots, \Omega_{x,N}$, we have $Y'_x(\QQ_p) \neq \emptyset$ by \eqref{eqn:Harari}. As explained in \S \ref{sec:BS}, it follows that the image $Q_p \in V_x(\QQ_p)$ of such a point satisfies $\alpha(Q_p) = 0$ for all $\alpha \in \mathcal{A}'$. In particular, for each $\alpha \in \mathcal{A}'$ we have
\begin{equation} \label{eqn:sum_0}
	\sum_{p \notin \Omega_{x,n+1},\dots, \Omega_{x,N}} \inv_p\alpha (Q_p)
	= \sum_{p \notin \Omega_{x,n+1},\dots, \Omega_{x,N}} 0 = 0.
\end{equation}
We now construct $p$-adic points $Q_p$ for the remaining primes $p$, to find an adele which is orthogonal to each $\alpha \in \mathcal{A}'$.

Fix $i=n+1,\dots, N$. First note that $p_i \in \Omega_{x,i}$ by our choice of $P$. For $p \in \Omega_{x,i}$, the fibre $V_x \bmod p$ is split. Thus by the Lang--Weil estimates, providing $S'$ is sufficiently large, there is a smooth $\FF_p$-point of $V_x$ which we can lift using Hensel's lemma to obtain a $\QQ_p$-point $Q_p$ of $V_x$. We do this for all $p \in \Omega\setminus\{p_i\}$. For $p_i$, we apply \cite[Lem.~9.20]{HW16} and the resulting arguments (cf.~\cite[(9.13)]{HW16} -- this uses the assumption that $p_i$ is completely split in $K_i$ and is the key step of Harari's trick). This yields the existence of a $\QQ_{p_i}$-point $Q_{p_i}$ such that
$$\inv_{p_i} \alpha(Q_{p_i}) = - \sum_{p \in \Omega_{x,i}\setminus \{ p_i\}} \inv_{p} \alpha(Q_p)$$
for all $\alpha \in \mathcal{A}'$.
Applying this to each $i$ and recalling \eqref{eqn:sum_0}, we find an adelic point $(Q_p) \in V_x(\Adele_\QQ)$ whose sum over all local invariants is trivial for each $\alpha$.
So for any $x$ satisfying \eqref{eqn:Harari}, we have  $V_x(\Adele_\QQ)^{\mathcal{A}'} \neq \emptyset$; in particular $V_x(\Adele_\QQ)^{\mathcal{A}} \neq \emptyset$.

However by \cite[Prop.~4.1]{HW16}, our  assumptions imply that $\Br V_\eta / \Br \kappa(\eta) \to \Br V_x / \Br \QQ$ is an isomorphism outside a thin set of $x \in \PP^1(\QQ)$. But, by a theorem of Serre  \cite[\S9.7]{Ser97}, only $O(B)$ of rational points in $\PP^1(\QQ)$ of height at most $B$ lie in any given thin set; thus \eqref{eqn:HW} still holds when restricted to  $x$ with the property that the image of $\mathcal{A}$ in $\Br V_x$ generates $\Br V_x / \Br \QQ$. For such $x$ we therefore have $V_x(\Adele_\QQ)^{\Br} \neq \emptyset$. But, by assumption, the Brauer--Manin obstruction is the only one to the Hasse principle for $V_x$, so $V_x(\QQ) \neq \emptyset$. This completes the proof.
\end{proof}

Theorem \ref{thm:RP} now follows immediately from Theorem \ref{thm:RP2}.

\begin{remark} \label{rem:admissible}
	The proof of Theorem \ref{thm:RP2} 
	shows that an admissible value of the exponent 
	$\omega(\pi)$ of $(\log B)^{-1}$ in the lower bound is
	$\Delta(\pi \circ \psi')$. In fact the proof gives exactly the value
	$\Delta(\pi \circ \psi)$ when
	$\pi \circ \psi$ has no new non-split fibres, i.e.~when
	$$
		Y_x \mbox{ split} \quad \iff \quad V_x \mbox{ split},
	\quad
	\mbox{for all closed points }x \in \PP^1.
	$$
\end{remark}

\section{Detector functions for general pencils} \label{sec:general}
We generalise our detector functions from \S\ref{sec:detectors} to general fibrations over $\PP^1$, i.e.~if there is a non-split fibre over a non-rational closed point.

\subsection{Set-up}
Let $V$ be a smooth projective variety over $\QQ$ equipped with a morphism $\pi:V \to \PP^1$ 	whose generic fibre is geometrically integral, such that each fibre of $\pi$ contains an irreducible component of multiplicity $1$. We assume that the fibre over some $y \in \PP^1(\QQ)$  is smooth and everywhere locally soluble.

Let $\Theta(\pi)$ be the set of closed points of $\PP^1$ which lie below the non-pseudo-split fibers of $\pi$. We let $U = \PP^1 \setminus \{\theta: \theta \in \Theta(\pi)\}$. Let $f_\theta(x_0,x_1) \in \ZZ[x_0,x_1]$ be a primitive binary  form whose zero locus in $\PP^1$ is $\theta$.
Let $S$ be a finite set of primes such that there exists a smooth proper scheme $\mathcal{V} \to \Spec \ZZ_S$
whose generic fibre is isomorphic to $V$, together with a morphism $\pi: \mathcal{V} \to \PP^1_{\ZZ_S}$ which extends the map 
$V \to \PP^1_{\QQ}$. We choose $S$ sufficiently large so that the fibre outside each $\theta \bmod p$ is pseudo-split and that the $\theta \bmod p$ are disjoint in $\PP^1_{\FF_p}$. ($\theta \bmod p$ is a collection of closed points of $\PP^1_{\FF_p}$ in general.)

\subsection{A negative result}
Analogously to Lemma \ref{lem:detector}, one might expect the existence of frobenian multiplicative functions $\varpi_\theta$ of mean $\delta_\theta(\pi)$ such that if $\prod_\theta \varpi_\theta(f_\theta(x_0,x_1)) = 1$, then $\pi^{-1}(x)$ has a $\QQ_p$-point for all sufficiently large primes $p$. In certain cases this holds.

\begin{example} \label{ex:too_simple!}
Consider the conic bundle surfaces
$$x^2 - ay^2 = f(t)$$
where $f$ is separable of even degree and $a \in \ZZ$ a non-square.
Then for $p \nmid 2a$, the condition that a fibre over $(t:1)$ with $p \| f(t)$ has a $\QQ_p$-point is that $a \in \QQ_p^{*2}$; a purely frobenian condition over $\QQ$.
\end{example}

This simple example is misleading; in general there are no such arithmetic functions, even for conic bundle surfaces 
(the next surface is a quartic del Pezzo).

\begin{lemma} \label{lem:DP4}
	Let $\pi:V \to \PP^1$ be the conic bundle surface given by a smooth compactification
	of
	$$x^2 - t y^2 = (t^2 -2 )z^2 \quad \subset \PP^2 \times \mathbb{A}^1.$$
	There is \textbf{no} arithmetic function $\varpi: \NN \to \{0,1\}$
	and \textbf{no} finite set of primes $S$ 
	with the following property: Let $p \notin S$ and $(t_0,t_1)$
	a primitive integer vector such that
	$p \| (t_0^2-2t_1^2)$. Then
	$\varpi(t_0^2-2t_1^2) = 1$ if and only if
	$\pi^{-1}(t_0:t_1)$ has a $\QQ_p$-point.
\end{lemma}
\begin{proof}
	Assume there exists $\varpi$ and $S$ as in the statement.
	Consider $\theta:t^2 - 2 = 0 \in \mathbb{A}^1_{\QQ}$.
	Let $p \equiv 7 \bmod 8$ with $p \notin S$.
	As $p \equiv 7 \bmod 8$ we have $2 \in \FF_p^{*2}$,
	so let $\alpha \in \FF_p$ be such that $\alpha^2 = 2$.
	Then we have $\theta \bmod p = (t- \alpha)(t + \alpha)$,
	and the fibres over these points are
	\begin{equation} \label{eqn:alpha}
		x^2 - \alpha y^2 = 0,\quad x^2 + \alpha y^2 = 0,
	\end{equation}
	respectively.	
	But 
	$
		\left(\frac{\alpha}{p}\right) \neq 
		\left(\frac{-\alpha}{p}\right)
	$
	since $-1 \notin \FF_p^{*2}$. Thus exactly one of $\pm \alpha$
	is in $\FF_p^{*2}$. Then
	\eqref{eqn:alpha} shows that the fibre over exactly one of 
	$t = \pm \alpha$ is split over $\FF_p$.

	Now let $t_0,t_1 \in \ZZ$ be such that $t_0^2 - 2t_1^ 2= p$
	(these exist by a classical theorem).
	Then $(t_0: \pm t_1) \bmod p$ are the points of
	$\theta \bmod p$, thus the fibre over exactly one is split
	(say the fibre over $(t_0:t_1) \bmod p$).
	A Hilbert symbol calculation shows that the fibre over
	$(t_0:t_1)$ has a $\QQ_p$-point, but the fibre over $(t_0:-t_1)$
	has no $\QQ_p$-point.
	However, by our assumptions on $\varpi$, we find that both 
	$\varpi(p)=\varpi(t_0^2 - 2t_1^2)  = 1$ and 
	$\varpi(p)=\varpi(t_0^2 - 2(-t_1)^2)  = 0$, which is a contradiction.	
\end{proof}

The problem above is the following: the condition $p \mid f_\theta(x)$ means that $x \bmod p \in \theta \bmod p$. But we don't know which closed point it corresponds to! The fibre over this closed point may or may not be split.

\subsection{The detector functions} \label{sec:ind_general}
One needs to work over the number field determined by $f_\theta$. Our choices are inspired by the constructions from \cite{BS16, HW16}. For simplicity of exposition, we assume that the fibre at infinity is smooth.

For each $\theta \in \Theta(\pi)$,  let $k_\theta = \QQ[x]/(f_\theta(x,1))$ and let  $\alpha_\theta$ denote the image of $x$ in $k_\theta$. Let $k_\theta \subset K_\theta$ be a finite Galois extension which contains the field of definition of every geometric irreducible component of $\pi^{-1}(\theta)$ and let $\Gamma_\theta = \Gal(K_\theta/k_\theta)$. We assume that $S$ contains all primes which ramify in the $K_\theta$. We identify the prime ideals of $k_\theta$ above $p$ with the irreducible factors of $f_\theta \bmod p$; in particular, we view these as closed points of $\PP^1_{\FF_p}$. We let
$$
	\mathcal{P}_\theta = \{ \fp \in S \} \cup \left\{ \fp \notin S : 
	\begin{array}{l}
		\Frob_\fp \in \Gamma_\theta \mbox{ fixes an irreducible component} \\
		\mbox{of $\pi^{-1}(\theta)$ of multiplicity $1$}
	\end{array}
	\right\}.
$$
Here $\fp$ is a non-zero prime ideal of the ring of integers of $k_\theta$. We abuse notation, and  write $\fp \in S$ if $\fp$  lies above a rational prime in $S$.
For $\theta \in \Theta(\pi)$ we define a completely multiplicative function $\varpi_\theta$ on the ideals of $k_\theta$ via
$$\varpi_\theta(\mathfrak{n}) =
\begin{cases}
	1, & \forall \, \mathfrak p \mid \mathfrak n \text{ we have } \mathfrak p \in \mathcal{P}_\theta \\
	0, & \text{otherwise}.
\end{cases}$$

The theory of frobenian (multiplicative) functions makes sense over any number field \cite[\S3.3]{Ser12}, and one immediately obtains the following.
\begin{lemma}
	Each $\varpi_\theta$ is a frobenian multiplicative function on the ideals of $k_\theta$ of mean $\delta_\theta(\pi)$.
\end{lemma}

We now have the following generalisation of Lemma \ref{lem:detector}.

\begin{lemma} \label{lem:detector_general}
	On enlarging $S$ if necessary, the following holds. Let $(x_0,x_1) \in \ZZ^2$ be such that $\gcd(x_0,x_1)=1$.
	If $\prod_{\theta \in \Theta(\pi)} \varpi_\theta(x_0 - \alpha_\theta x_1) = 1$
	then $\pi^{-1}(x_0:x_1)$ has a $\QQ_p$-point for all $p \notin S$.
\end{lemma}
\begin{proof}
	Let $x_0,x_1$ be such that $\prod_{\theta \in \Theta(\pi)} \varpi_\theta(x_0 - \alpha_\theta x_1) = 1$ and $\gcd(x_0,x_1)=1$.
	Let $p \notin S$. We claim that $(x_0:x_1) \bmod p$ lies below
	a split fibre. 
	
	If $p \nmid \prod_{\theta \in \Theta(\pi)} f_\theta(x_0,x_1)$
	then $(x_0:x_1) \bmod p \notin \theta \bmod p$ for 
	all $\theta \in \Theta(\pi)$, thus the fibre is split.
	If $p \mid f_\theta(x_0,x_1)$ for some $\theta \in \Theta(\pi)$
	then $(x_0:x_1) \bmod p \in \theta \bmod p$,
	so $(x_0:x_1) \bmod p$ corresponds to the prime ideal $\fp =(x_0 - \alpha_\theta x_1,p)$ of $k_\theta$.
	As $\fp \mid (x_0 - \alpha_\theta x_1)$ and 
	$\varpi_\theta(x_0 - \alpha_\theta x_1) = 1$, we find that
	$\fp \in \mathcal{P}_\theta$.
	Hence $\Frob_\fp$ fixes an irreducible component of multiplicity
	one of the fibre, so the fibre over $(x_0:x_1) \bmod p$ is  split, as required.
	
	The result now follows from the Lang--Weil estimates and
	Hensel's lemma, on enlarging $S$ if necessary.
\end{proof}

One deals with the small primes and the real place exactly as in \S\ref{sec:small_primes}. We deduce the following.

\begin{corollary} \label{cor:ind_general}
	There exists a finite set of primes $S$ and $\delta > 0$ such that
	$$
	N_{\mathrm{loc}}(\pi,B) \geq \frac{1}{2}  \sum_{\substack{(x_0,x_1) \in \ZZ^2   \\ |x_0|,|x_1| \leq B  \\ \gcd(x_0,x_1) = 1 
	 \\  |x_1/x_0 - y_1/y_0|_v < \delta \, \forall v \in S \cup \{\infty\} }} 
	 \prod_{\theta \in \Theta(\pi)} \varpi_\theta(x_0 - \alpha_\theta x_1).
	$$
\end{corollary}

Note that it is still linear forms which are used for the detector functions. But the linear forms $x_0 - \alpha_\theta x_1$ are defined over the larger field $k_\theta$, not over $\QQ$.

\begin{example}
	(1) We show that our detector functions recover the naive ones
	for the conic bundle surfaces
	$x^2 - ay^2 = f(t)z^2$
	from Example~\ref{ex:too_simple!}. For simplicity 
	assume that $f$ is irreducible. Let $\varpi_\theta$ be as in 
	\S \ref{sec:ind_general}. Explicitly, for almost all $\fp$
	$$\varpi_\theta(\fp) = 1 \quad \iff \quad \left(\frac{a}{\fp}\right)
	= 1$$
	where the Legendre symbol is over $k_\theta$. However, let $\varpi$
	be the naive detector function over $\QQ$, where for almost all $p$
	we have
	$$\varpi(p) = 1 \quad \iff \quad \left(\frac{a}{p}\right)
	= 1.$$	
	For \emph{prime ideals $\fp \mid p$ of degree $1$} we have
	$\varpi_\theta(\fp) = \varpi(p);$
	indeed, the  map $\ZZ \to \FF_\fp$ induces a canonical
	isomorphism $\FF_\fp \cong \FF_p$. As $a \in \ZZ$, our condition is
	independent of the choice of the prime ideal $\fp$
	of degree $1$. 
	
	So let $x_0,x_1$ be such that $\gcd(x_0,x_1) = 1$
	and let $F$ be the homogenisation of $f$.
	Note that $F(x_0,x_1) = \Norm_{k_\theta/\QQ}(x_0 - \alpha_\theta x_1)$.
	It easily follows that the ideal $(x_0 - \alpha_\theta x_1)$ may
	only be divisible by prime ideals of degree $1$, and that
	$p \mid F(x_0,x_1)$ if and only if $x_0 - \alpha_\theta x_1$ 
	is divisible by some prime
	ideal of degree $1$ over $p$. We conclude that
	$\varpi_\theta( x_0 - \alpha_\theta x_1) = \varpi(F(x_0,x_1)),$
	which recovers Example \ref{ex:too_simple!}.

	(2) We next compute our detector functions for the conic bundle surfaces from Lemma \ref{lem:DP4}.	The non-split fibres
	occur over the closed points $\theta_1 : t = 0$ and $\theta_2 : t^2 - 2 = 0$, with corresponding frobenian functions
	\begin{align*}
		\varpi_1(p) = 1 \iff  \left(\frac{-2}{p}\right) = 1,
		\quad
		\varpi_2(\fp) = 1 \iff  \left(\frac{\sqrt{2}}{\fp}\right)
		=1,
	\end{align*}
	for $\fp$ a prime of $\QQ(\sqrt{2})$. The 
	summand in Corollary \ref{cor:ind_general} is
	$\varpi_1(x_0)\varpi_2(x_0 - \sqrt{2}x_1)$.
	Note that if $p \equiv 7 \bmod 8$ and $p = \fp \fq$, then we have
	$$\left(\frac{\sqrt{2}}{\fp}\right) = 
	\left(\frac{-\sqrt{2}}{\fq}\right) = -
	\left(\frac{\sqrt{2}}{\fq}\right)$$
	Hence
	$\varpi_2(\fp) \neq \varpi_2(\fq)$ so here $\varpi_2$ is not constant
	on prime ideals of degree $1$ above $p$.
	This agrees with the behaviour observed
	in the proof of Lemma \ref{lem:DP4}.
\end{example}

For sums of multiplicative functions as in Corollary \ref{cor:ind_general}, one expects that the asymptotic behaviour is controlled by expressions of the form
$$\sum_{\Norm \mathfrak a \leq x} \frac{\mu^2(\fa) \varpi(\fa)}{\Norm \fa}$$
(see \cite{BS17} for upper bounds of this shape). The following is a minor variant of the results from \S \ref{sec:frob}, and agrees with the conjectural lower bound.

\begin{lemma} 
	$$\sum_{\Norm \mathfrak a \leq x} \frac{\mu^2(\fa) \varpi(\fa)}{\Norm \fa}
	\asymp (\log B)^{m(\varpi)}.$$
\end{lemma}

\section{Brauer groups} \label{sec:Brauer}
In this section we prove Theorem \ref{thm:Brauer}. We will use the various properties of Brauer groups recalled in \S \ref{sec:Brauer_basic}.

\subsection{Specialisations and ramification}
The following will be used to construct the detector functions in the proof of Theorem \ref{thm:Brauer}.

\begin{proposition} \label{prop:Brauer_integral}
	Let $Y$ be a smooth geometrically integral variety over a number field $k$
	and let $b \in \Br k(Y)$. Then there exists a finite set 
	of primes $S$ of $k$ together with a regular model $\mathcal{Y}$ for $Y$ over $\OO_{k,S}$ such that the following holds.
	
	Let $v \notin S$ and assume that $b \otimes k_v$ is unramified at all codimension $1$ points of $Y_{k_v}$.
	Then for all $y \in \mathcal{Y}(\OO_v)$ we have $b(y) = 0 \in \Br k_v$.
\end{proposition}
If $b$ is in fact unramified at all codimension $1$ points of $Y$ (so that $b \in \Br Y$ by \eqref{seq:purity}),
then it is well-known that for all but finitely
many places $v$ we have $b(y) = 0$ for all $y \in \mathcal{Y}(\OO_v)$ \cite[Prop.~8.2.1]{Poo17}. Proposition \ref{prop:Brauer_integral} obtains
a generalisation of this to the case when $b$ may be ramified on $Y$.

\begin{proof}[Proof of Proposition \ref{prop:Brauer_integral}]
	By \eqref{eqn:ell}, it suffices to prove the result when  $b$ 
	has order  power of a prime $\ell$.
	Choose a finite set of primes $S$ such that $\ell \in \OO_{k,S}^*$,
	together with a regular integral model $\mathcal{Y}$ 
	for $Y$ over $\OO_{k,S}$. Enlarging $S$ if necessary, we extend $b$ to an element of 
	some open subset $\mathcal{U} \subset \mathcal{Y}$ such that $\mathcal{U} \to \Spec \OO_{k,S}$
	is surjective. Thus $b \otimes \FF_v \in \Br \mathcal{U}_{\FF_v}$ is well-defined for all $v \notin S$.
	Let $U=\mathcal{U} \cap Y$; by \eqref{seq:purity} we may assume that the complement of $U$ in $Y$ is pure of codimension $1$.
	We also assume that $\mathcal{Y} \otimes \FF_v$ is geometrically integral for all $v \notin S$.

	Now let $v \notin S$ be such that $b \otimes k_v$ is unramified at all codimenison $1$ points
	of $Y_{k_v}$. We claim that $b$ is also unramified at all codimension $1$ points of
	$\mathcal{Y}_{\OO_v}$. To see this, let $\mathcal{D}$ be an irreducible divisor of $\mathcal{Y}_{\OO_v}$.
	If $\mathcal{D} = \mathcal{Y}_{\FF_v}$,
	then $b$ is unramified along $\mathcal{D}$; indeed, by construction $b$ is well-defined on the non-empty
	open subset $\mathcal{Y}_{\FF_v} \cap \mathcal{U}$ of $\mathcal{Y}_{\FF_v}$.	
	Assume instead that $\mathcal{D}$ meets the generic fibre
	in some divisor $D$.
	The residue maps then give rise to the commutative diagram
	\[\xymatrix{
	\Br \kappa(\mathcal{Y}_{\OO_v})\{\ell\}  \ar[d] \ar[r]^{\res_{\mathcal{D}}\hspace{10pt}} & \HH^1(\kappa(\mathcal{D}), \QQ_\ell/\ZZ_\ell) \ar[d]  \\ 
	\Br \kappa(Y_{k_v}) \{\ell\} \ar[r]^{\res_D\hspace{10pt}} & \HH^1(\kappa(D), \QQ_\ell/\ZZ_\ell).
	} \] 
	However the maps $\kappa(\mathcal{Y}_{\OO_v}) \to \kappa(Y_{k_v})$ and 
	$\kappa(\mathcal{D}) \to \kappa(D)$ are isomorphisms,
	thus the downward maps are also isomorphisms. As $\res_D(b) = 0$ by assumption,
	we find that $\res_{\mathcal{D}}(b) = 0$. This proves the claim,
	hence $b \otimes k_v \in \Br \mathcal{Y}_{\OO_v}$
	by \eqref{seq:purity}.
	
	We may now prove the proposition. Let $y \in \mathcal{Y}(\OO_v)$. As $b \otimes k_v \in \Br \mathcal{Y}_{\OO_v}$,
	we have  $b(y) \in \Br \OO_v = 0$ (see \cite[Cor.~6.9.3]{Poo17}), as required.
\end{proof}

We now obtain a quantitative description of those primes $v$ which satisfy the assumptions of Proposition \ref{prop:Brauer_integral}.
For simplicity, we only consider geometrically irreducible divisors.

\begin{proposition} \label{prop:frob_residue}
	Let $Y$ be a smooth geometrically integral
	variety over a number field $k$ and let $\br \subset \Br k(Y)$ be a finite subgroup.
	Let $D \in Y^{(1)}$ and assume that $k$ is algebraically closed in the residue field $k(D)$.
	Then the set of places
	$$F_D(b) := \{v \in \Val(k) : \res_D( \br )\otimes k_v = 0\}$$ 
	is frobenian. Moreover:
	
	\begin{enumerate}
		\item If $\res_D(\br) \otimes \bar{k} \neq 0$ then $F_D(b) = \emptyset$. 
		\item If $\res_D(\br) \otimes \bar{k} = 0$ then 
		$\dens(F_D(b)) = 1 / |\res_D(\br)|.$
	\end{enumerate}
\end{proposition}
	In the statement, for a field extension $k \subset L$ we let 
	$\cdot \otimes_k L: \mathrm{H}^1(k(D), \QQ/\ZZ) \to \mathrm{H}^1(k(D) \otimes_k L, \QQ/\ZZ)$
	denote the usual restriction map on Galois cohomology.
\begin{proof}
	The group of residues $\res_D(\br)$ is a subgroup of $\mathrm{H}^1(k(D), \QQ/\ZZ)$. Thus it determines some finite abelian field extension $k(D) \subset R$ of degree $\res_D(\br)$.

	First assume that $\res_D(\br) \otimes \bar{k} \neq 0$. 
	We need to show that $\res_D(\br) \otimes k_v \neq 0$
	for all places $v$.
	To do this, it suffices to show that
	these residues are non-zero after a finite field extension. 
	In particular we may pass to a finite field extension if we wish,
	and assume that $k$ is algebraically closed in $R$.
	In this case $D$ is geometrically irreducible,
	the field $R$ is the function field of a geometrically irreducible variety over $k$, 
	and $k(D) \subset R$ is a non-trivial finite field extension.
	It follows that $k(D) \otimes k_v \subset R \otimes k_v$ is still a non-trivial
	finite field extension for all $v$. However, this is exactly the field extension
	corresponding to the residues $\res_D(\br) \otimes k_v$; this group is therefore
	non-trivial, as required.

	Now assume that $\res_D(\br) \otimes \bar{k} = 0$. Inflation-restriction
	yields
	$$0 \to \HH^1(\Gal(\bar{k}(D)/k(D)),\QQ/\ZZ) \to \HH^1( k(D),\QQ/\ZZ) \to \HH^1( \bar{k}(D),\QQ/\ZZ);$$
	thus $k(D) \to R$ is the base-change of some 
	finite abelian extension $k \subset K$ of number fields of degree $\res_D(\br)$.
	A moment's thought reveals that
	$$F_D(b) = \{ v \in \Val(k) : v \text{ is completely split in } K\}. $$
	The set of such places is clearly frobenian of density 
	$1/[K:k] = 1 / |\res_D(\br)|$.
\end{proof}

\subsection{Proof of Theorem \ref{thm:Brauer}}
The upper bound is obtained in \cite[\S 5.3]{LS16}. It therefore suffices to prove the lower bound. 

Let $S$ be a sufficiently large set of primes.
Let $U$ and $\br$ be as in Theorem~\ref{thm:Brauer} and $y \in U(\QQ)_\br$.
Let $\Theta$ denote the set of codimension $1$ points of $\PP^n$ which lie outside of $U$. For each $D \in \Theta$, let $L_D \in \ZZ[x_0,\dots,x_n]$ be the primitive linear form defining $D$ and let $\mathcal{P}_{D}$ be the union of $S$ and those primes $p$ for which $\res_D \br \otimes \QQ_p = 0$. As $\br \subset \Br_1 U$ and each $D$ is geometrically integral, Proposition \ref{prop:frob_residue} implies that $\mathcal{P}_{D}$ is frobenian of density $1/|\res_D(\br)|$. 
We define the completely multiplicative function $\varpi_D$ via
$$\varpi_D(n) =
\begin{cases}
	1, & \forall p \mid n \text{ we have } p \in \mathcal{P}_D \\
	0, & \text{otherwise}.
\end{cases}$$

\begin{lemma}
	Each $\varpi_D$ is a frobenian multiplicative function of mean $1/|\res_D(\br)|$.
\end{lemma}
\begin{proof}
	Follows immediately from Proposition \ref{prop:frob_residue}.
\end{proof}

These functions enjoy an analogue of Lemma \ref{lem:detector}.
\begin{lemma} \label{lem:Brauer_large_primes}
	Enlarging $S$ if necessary, the following holds.
	Let $\xx=(x_0,\ldots,x_n)$ be a primitive integer vector.
	If $\prod_{D \in \Theta} \varpi_D(L_D(\xx)) = 1$
	then $b(x_0:\cdots:x_n) \otimes \QQ_p = 0 \in \Br \QQ_p$ for all $p \notin S$ and all $b \in \br$.
\end{lemma}
\begin{proof}
	Let $p$ be a prime and 
	let $\Theta_p(\xx)$ be the subset of $\Theta$ of those $D$ for which $p \mid L_D(\xx)$.
	Let $\mathcal{Y} = \PP^n_{\ZZ} \setminus (\cup_{D \notin \Theta_p(\xx)} \mathcal{D})$,
	where $\mathcal{D}$ is the closure of $D$ in $\PP^n_{\ZZ}$.
	Our choice of $\Theta_p(\xx)$ implies that $x \in \mathcal{Y}(\ZZ_p)$. As $\varpi_D(p) = 1$ for all $D \in \Theta_p(\xx)$,
	we have $p \in \cap_{D \in \Theta(\xx)}\mathcal{P}_D$. Thus, by definition, 
	the Brauer elements $\br \otimes \QQ_p$ are unramified at each $D \in \Theta_p(\xx)$. 
	If $p$ is sufficiently large (independently of $\xx$),
	it now follows from Proposition~\ref{prop:Brauer_integral} that $b(x) \otimes \QQ_p = 0$
	for all $b \in \br$, as required.
\end{proof}

We next obtain an analogue of Lemma \ref{lem:conclusion}.
\begin{lemma} \label{lem:Brauer_bound}
	There exists $\delta > 0$ such that
	$$N(U,\br,B) \geq \frac{1}{2} \sum_{\substack{(x_0,\ldots,x_n) \in \ZZ^{n+1} 
	\\ \gcd(x_0,\ldots,x_n) = 1  \\ \max_i |x_i| \leq B \\ 
	\max\limits_{\mathclap{v \in S \cup \{\infty\}}} \,  |x_i/x_0 -y_i/y_0|_v < \delta }} 
	 \prod_{D \in \Theta} \varpi_D(L_D(x_0,\ldots,x_n)).$$
\end{lemma}
\begin{proof}
	Let $\xx=(x_0,\ldots,x_n)$ be a primitive integer vector with the property
	$\prod_{D \in \Theta} \varpi_D(L_D(x_0,\ldots,x_n)) =1$.
	By Lemma \ref{lem:Brauer_large_primes}, we have $b(x) \otimes \QQ_p = 0$ for all primes $p \notin S$, where $x=(x_0:\cdots:x_n)$.
	
	For the real place and small primes recall that $y \in U(\QQ)_\br$.
	As the Brauer pairing is locally constant for the real and $p$-adic topologies \cite[Prop.~8.2.9]{Poo17},
	we deduce the existence of $\delta > 0$ such that if 
	$|x_i/x_0 -y_i/y_0|_v < \delta$
	for each $i \neq 0$ and each $v \in S \cup \infty$, then $b(x) \otimes \QQ_v = 0$
	for all $v \in S \cup \{\infty\}$ and all $b \in \br$.
	
	For $\xx$ as in the sum, we have shown that $b(x) \otimes \QQ_v = 0$
	for all places $v$ of $\QQ$ and all $b \in \br$. However 
	the Hasse principle for $\Br \QQ$ \eqref{eqn:CFT}
	implies that $b(x) = 0 \in \Br \QQ$ for all $b \in \br$, as required.
\end{proof}

Given Lemma \ref{lem:Brauer_bound}, we see that Theorem \ref{thm:Brauer} follows from Theorem \ref{thm:frob_applications}. \qed

\subsection{A negative result}
We finish this section by highlighting some of the subtleties which arise if one is trying to generalise the proof of Theorem \ref{thm:Brauer} to the case $\br \subset \Br U$, i.e. where the Brauer group elements can be  transcendental. Here we  have a transcendental analogue of Lemma \ref{lem:DP4}.

\begin{example} \label{ex:Hooley}
	Consider the conic bundle
	\begin{equation} \label{eqn:Hooley}
		a_0x^2 + a_1x_1^2 + a_2x_2^2 = 0.
	\end{equation}
	There is no arithmetic function $\varpi$ with the following properties:
	Let $p$ be an odd prime and $(a_0,a_1,a_2)$ a primitive integer vector such that
	$p \| a_0$ but $p \nmid a_1a_2$. Then
	$\varpi(a_0) = 1$ if and only if the conic \eqref{eqn:Hooley} has a $\QQ_p$-point.
\end{example}

This is proved without  difficulty.
The papers \cite{LTT17,Lou13} also restrict to algebraic Brauer group elements, as the transcendental case is more difficult in general. The only transcendental cases known are the lower bounds obtained by Hooley in \cite{Hoo93, Hoo07}, which includes the correct lower bound for Example \ref{ex:Hooley}. It would be interesting to try to solve Serre's problem for other transcendental cases.

\section{Multinorms} \label{sec:norms}
We now prove Theorem \ref{thm:norms}. We let $V,W$ and $E$ be as in the statement
of Theorem \ref{thm:norms}. Let $\psi:W \to \PP^n$ be the projection given by the 
$x$-coordinate. By \cite[Lem.~5.2]{LS16} we have
$$N_{\mathrm{loc}}(\psi,B) = N_{\mathrm{loc}}(\pi,B) + O_\varepsilon(B^{n + 1/2 + \varepsilon})$$
for any $\varepsilon > 0$. Thus to prove a lower bound we may work with the explicit equation \eqref{eqn:multinorm}
for $W$. 
Theorem \ref{thm:norms} concerns rational numbers, so we first pass to a homogeneous problem
involving integers.	Let
\begin{equation}
	e = \gcd\{ [E_i: \QQ]: i = 1, \dots, s\},
\end{equation}
i.e.~$e$ is the $\gcd$ of the degrees of the maximal subfields $E_i$ of the finite \'etale $\QQ$-algebra $E$.
We let $\mathbf{L}_j$
be the homogenisation of the linear polynomial $L_j$. We also let 
$\mathbf{L}_0 = x_0$ and let $a_0 \in \ZZ$ be a representative of the congruence class $-r \bmod e$. 
A moment's thought reveals the following.

\begin{lemma} \label{lem:moment}
	Let $\xx \in \ZZ^{n+1}$ be a primitive integer vector with $x_0 \neq 0$. Then  $\psi^{-1}(x_1/x_0,\dots,x_n/x_0)$ is everywhere locally soluble if and only if
	$ \mathbf{L}_0(\xx)^{a_0}\cdots \mathbf{L}_r(\xx)^{a_r}$
	is a norm from $\prod_{i=1}^r \Adele_{E_i}$.
\end{lemma}

Thus we  need to understand when a $p$-adic number is a norm from a product of finite field extensions.
This is achieved by the following simple lemma.

\begin{lemma} \label{lem:local_norm}
	Let $F$ be a finite \'{e}tale $\QQ_p$-algebra such that the integral closure $\mathcal{O}_F$ of $\ZZ_p$
	in $F$ is unramified over $\ZZ_p$.
	Let 
	$f = \gcd_{k \subset F} \,\, [k : \QQ_p],$
	where the greatest common divisor is taken over all maximal subfields $k$ of $F$.
	Then an element $x \in \QQ_p$ is a norm from $F$ if and only if $f \mid v_p(x)$.
\end{lemma}
\begin{proof}
	We write $F$ as a product of its maximal subfields,
	and $p$ is a uniformiser in each of these subfields
	as $\mathcal{O}_F$ is unramified over $\ZZ_p$. This shows that
	$$\{ v_p(N_{F/\QQ_p}(y)) : y \in F\} = f\ZZ$$
	as ideals of $\ZZ$. In particular if $f \nmid v_p(x)$ then $x$ is clearly not a norm from $F$.
	
	So assume that $f \mid v_p(x)$. Then as $p^f$ is a norm from $F$, it suffices to show that
	all units of $\ZZ_p$ are norms from $F$. However this follows from the fact that $\mathcal{O}_F$ is unramified over $\QQ_p$
	\cite[Prop.~V.2.3]{Ser79}.
\end{proof}

We each $j =0,\dots,r$ we therefore let
$$\mathcal{P}_j = S \cup \left\{ \text{primes } p: \gcd_{ \mathclap{k_p \subset E_p}} \,\, [k_p : \QQ_p] \text{  divides } a_j\right\},$$
where the greatest common divisor is taken over  all maximal subfields $k_p$ of the finite \'{e}tale $\QQ_p$-algebra
$E_p = E \otimes_\QQ \QQ_p$ and $S$ is the set of primes which are ramified in $E$. Let
$$\varpi_j(n) =
\begin{cases}
	1, & \forall p \mid n \text{ we have } p \in \mathcal{P}_j \\
	0, & \text{otherwise}.
\end{cases}$$
This is easily seen to be a frobenian multiplication function with $m(\varpi_j) \neq 0$.

\begin{lemma}
	Let $\xx=(x_0,\ldots,x_n)$ be a primitive integer vector with $x_0 \neq 0$.
	Suppose that $\prod_{j=0}^r \varpi_j(\mathbf{L}_j(\xx)) = 1.$
	Then $\psi^{-1}(x)$ has a $\QQ_p$-point for all $p \notin S$.
\end{lemma}
\begin{proof}
	Let $p \notin S$ and let $j \in \{0,\dots,r\}$ with $\varpi_j(\mathbf{L}_j(\xx)) = 1$.
	First suppose that $p \nmid \mathbf{L}_j(\xx)$. Then $\mathbf{L}_j(\bx)^{a_j}$ is a $p$-adic unit
	hence a local norm  by Lemma \ref{lem:local_norm}. Next suppose that 
	$p \mid \mathbf{L}_j(\xx)$. Then as $\varpi_j(p)=1$ we have $p \in \mathcal{P}_j$.
	As  $a_j \mid v_{p}(\mathbf{L}_j(\xx)^{a_j})$, 
	it follows from Lemma \ref{lem:local_norm} and the 
	choice of $\mathcal{P}_j$ that $\mathbf{L}_j(\xx)^{a_j}$ is a local norm.
	
	Thus when $\prod_{j=0}^r \varpi_j(\mathbf{L}_j(\xx)) = 1$ we see that each 
	 $\mathbf{L}_j(\xx)^{a_j}$ is a local norm.
	 To finish, it suffices to note that the product of norms is again a norm.
\end{proof}

This takes care of the primes not in $S$. For small primes and the real place one proceeds
in an analogous manner to \S \ref{sec:small_primes}. An application
 of Theorem \ref{thm:frob_applications} then completes the proof of the following
more explicit version of Theorem \ref{thm:norms}.

\begin{theorem} \label{thm:norms_2}
	In the above notation and the notation of Theorem \ref{thm:norms} we have
	$$N_{\mathrm{loc}}(\pi,B) \gg B^{n+1}\prod_{j=0}^r(\log B)^{\dens(\mathcal{P}_j)-1}.$$
\end{theorem}

To complete the proof of Theorem \ref{thm:norms} it suffices to prove the following.

\begin{lemma} \label{lem:norms_finish}
	We have
	\begin{equation}	
		\dens(\mathcal{P}_j) = \delta_{D_j}(\pi) , \quad j \in \{0,\dots,r\},
	\end{equation}
	where $D_j$ is the hyperplane in $\PP^n$ determined by $\mathbf{L}_j$.
\end{lemma}
\begin{proof}
 	The proof is inspired by the proof of \cite[Thm.~5.5]{LS16}.
 	Choose a finite Galois extension $k/ \QQ$ which contains the splitting fields of the $E_i$.
 	Let $\Gamma = \Gal(k/\QQ)$. Then the Chebotarev density theorem implies that
 	$$\delta_{D_j}(\pi) = \dens\left( \text{primes } p: 
 	\begin{array}{ll}
 		\Frob_p \in \Gal(k/\QQ) \text{ fixes some multiplicity $1$} \\
 		\text{geometric irreducible component of }\pi^{-1}(D_j)
 	\end{array}
 	\right).$$
 	However $\Frob_p$ fixes some multiplicity $1$ geometric irreducible
 	component if and only if the fibre $\pi^{-1}(D_j)$ is split over $\QQ_p$.
 	As the divisors $D_j$ are geometrically integral, \cite[Thm.~5.4]{LS16} shows
 	that this happens if and only if 
 	$\gcd_{k_p \subset E_p} [k_p : \QQ_p]$ divides $a_j,$
 	where the greatest common divisor is  over  all maximal subfields $k_p$ of 
 	$E_p = E \otimes_\QQ \QQ_p$. The lemma now follows from the definition of $\mathcal{P}_j$.
\end{proof}

This completes the proof of Theorem \ref{thm:norms}. \qed

\section{Multiple fibres} \label{sec:multiple}

We finish with the proof Theorem \ref{thm:double}, using the method from \cite[\S 2]{CTSSD97}.
Let $V$ be a smooth projective variety over a number field $k$
equipped with a morphism $\pi:V \to \PP^1$
whose generic fibre is geometrically integral. 
We assume that $\pi$ has at least $6$ double fibres over $\bar{k}$. 
(We say that $\pi$ has a \emph{double fibre} over a point $P \in \PP^1$
if $\pi^*P = 2D$ for some divisor $D$ on $V$.)

To prove the result, we may assume that the fibre at infinity is smooth. Moreover, we are 
free to pass to a finite field extension of $k$, so that we may assume that every 
double fibre over $\bar{k}$ is actually defined over $k$.

Choose a squarefree polynomial $f \in k[x]$ of degree $6$
such that the fibre over every root of $f$ is a double fibre.
For $\alpha \in k^*/k^{*2}$
we denote by $C_\alpha$ the hyperelliptic curve  $a y^2 = f(x)$,
for some representative $a \in k^*$ of $\alpha$. Let
$\mathcal{T}_{\alpha}$ be the normalisation of $V \times_{\PP^1} C_\alpha$, so that we obtain the commutative diagram
\begin{equation} \label{eqn:torsor}
\begin{split}
\xymatrix{
\mathcal{T}_{\alpha} \ar[d]^{\pi_\alpha} \ar[r]^{\tau_\alpha} & V \ar[d]^\pi  \\ 
C_{\alpha} \ar[r]^{w_\alpha} & \PP^1.
}
\end{split}
\end{equation}
By \cite[Rem.~2.1.1]{CTSSD97} the map $\tau_\alpha$ is a $\ZZ/2\ZZ$-torsor. Moreover, let $U \subset \PP^1$ be the complement of the singular locus of $\pi$. Then the diagram \eqref{eqn:torsor} is cartesian on restricting to $U$, since the fibre product is smooth above $U$.

Let now $S$ be a finite set of places of $k$ and let $k^S = \prod_{v \notin S} k_v$.
Let $x \in U(k) \cap \pi(V(k^S))$. Choose $\alpha \in k^*/k^{*2}$ such that $x \in w_{\alpha}(C_{\alpha}(k))$
(e.g.~$\alpha = f(x)$). As the diagram \eqref{eqn:torsor} is cartesian over $U$,
we see that the fibre $\pi_{\alpha}^{-1}(w_{\alpha}^{-1}(x))$ has a $k^S$-point.
In particular $\mathcal{T}_\alpha(k^S) \neq \emptyset$.

As $V$ is projective and $\tau_\alpha$ is a $\ZZ/2\ZZ$-torsor, there exists a finite set $A_S \subset k^*/k^{*2}$
such that $\mathcal{T}_\alpha(k^S) = \emptyset$ for all $\alpha \notin A_S$ (see \cite[Prop.~5.3.2]{Sko01}).
It follows that 
$$
	\{x \in \PP^1(k): x \in \pi(V(k^S))\}  \subset
	(\PP^1 \setminus U)(k) \bigcup_{\alpha \in A_S} \{x \in \PP^1(k): x \in w_\alpha(C_{\alpha}(k))\}.
$$
As $\deg f = 6$, each $C_\alpha$ is a hyperelliptic
curve of genus $2$. Therefore each $C_\alpha(k)$ is finite by Faltings's theorem \cite{Fal83}.
Theorem \ref{thm:double} follows. \qed

\end{document}